\documentclass[12pt,reqno]{amsart}
\usepackage{fullpage}
\usepackage{amsfonts}
\usepackage{amssymb}
\usepackage{enumerate}
\usepackage{times}
\usepackage{graphicx}
\usepackage{url}
\usepackage{eucal}
\usepackage{epstopdf}
\usepackage{mathrsfs}
\usepackage{amsmath}
\usepackage{xcolor}
\usepackage{mathrsfs}
\usepackage{enumerate}
\usepackage{hyperref}
\usepackage{ccaption}
\usepackage{dsfont}
\usepackage{soul}
\DeclareFontFamily{OT1}{rsfs}{}
\DeclareFontShape{OT1}{rsfs}{n}{it}{<-> rsfs10}{}
\DeclareMathAlphabet{\mathscr}{OT1}{rsfs}{n}{it}
\usepackage{enumitem} 
\usepackage{yfonts}
\usepackage{bbm}

\newtheorem{theorem}{Theorem}[section]
\newtheorem{prop}[theorem]{Proposition}
\newtheorem{corollary}[theorem]{Corollary}
\newtheorem{lemma}[theorem]{Lemma}
\newtheorem*{conj}{Conjecture}

{
\theoremstyle{remark}
\newtheorem*{remark}{Remark}
}

\renewcommand{\mod}[1]{{\ifmmode\text{\rm\ (mod~$#1$)}\else\discretionary{}{}{\hbox{ }}\rm(mod~$#1$)\fi}}

\newcommand{\R}{{\mathbb R}}

\newcommand{\Lo}{\mathscr{L}}

\newcommand{\chiqq}{ \sideset{}{^*}\sum_{\psi \mod{q} }}

\newcommand{\e}{\varepsilon}
\newcommand{\pl}{\text{Li}}

\begin{document}

\title[Discrete mean value]
 {A discrete mean value  of the Riemann zeta function} 

 \thanks{Nathan Ng was supported by the NSERC discovery grant  RGPIN-2020-06032. K\"{u}bra Benli was supported by a Pacific Institute for the Mathematical Sciences (PIMS) postdoctoral fellowship at the University of Lethbridge and also by the joint FWF-ANR project Arithrand: FWF: I 4945-N and ANR-20-CE91-0006.
 Ertan Elma was supported by a University of Lethbridge postdoctoral fellowship. 
 This research was also funded by the PIMS Collaborative Research Group {\it $L$-functions in Analytic Number Theory}.}

\date{\today}



\author[K\"{u}bra Benl\.{i}]{K\"{u}bra Benl\.{i}}
\address{University of Lethbridge \\ Department of Mathematics and Computer Science \\ 4401 University Drive \\ Lethbridge, AB, Canada \ T1K 3M4 }
\email{kubra.benli@uleth.ca}

\author[Ertan Elma]{Ertan Elma}
\address{University of Lethbridge \\ Department of Mathematics and Computer Science \\ 4401 University Drive \\ Lethbridge, AB, Canada \ T1K 3M4 }
\email{ertan.elma@uleth.ca}

\author[Nathan Ng]{Nathan Ng}
\address{University of Lethbridge \\ Department of Mathematics and Computer Science \\ 4401 University Drive \\ Lethbridge, AB, Canada \ T1K 3M4 }
\email{nathan.ng@uleth.ca}

\begin{abstract}
In this work, we estimate the sum 
\begin{align*}
    \sum_{0 < \Im(\rho) \leq T} \zeta(\rho+\alpha)X(\rho) Y(1\!-\! \rho)
\end{align*}
 over the nontirival zeros $\rho$ of the Riemann zeta funtion where $\alpha$ is a complex number with $\alpha\ll 1/\log T$ and $X(\cdot)$ and $Y(\cdot)$ are some Dirichlet polynomials.   
Moreover, we estimate the discrete mean value above for higher derivatives where $\zeta(\rho+\alpha)$ is replaced by $\zeta^{(m)}(\rho)$ for all $m\in\mathbb{N}$.  The formulae we obtain generalize a number of previous results in the literature. As an application, assuming the Riemann Hypothesis we obtain the lower bound
\begin{align*}
		 \sum_{0 < \Im(\rho) < T} | \zeta^{(m)}(\rho)|^{2k} \gg T(\log T)^{k^2+2km+1} \quad \quad (k,m\in\mathbb{N})
	\end{align*}
 which was previously known under the Generalized Riemann Hypothesis, in the case $m=1$.
\end{abstract}

\maketitle


\newcommand{\kubra}[1]{{\color{blue} \sf  Kubra: #1}}
\newcommand{\nathan}[1]{{\color{purple} \sf  Nathan: #1}}
\newcommand{\ertan}[1]{{\color{red} \sf  Ertan: #1}}

\section{Introduction}

In this work, we study a discrete mean value of the Riemann zeta function. 
There has been extensive research on these types of sums as they have many interesting number theoretic applications  including to the proportion of simple zeros of the 
zeta function (see \cite{CGG}, \cite{BHB}), large and small gaps between the zeros of the zeta function
(see \cite{CGG2}, \cite{BMN2}, \cite{CGG3}, \cite{Ng}), and the distribution of the  sum of the M\"{o}bius function and related functions (see  \cite{Hum}, \cite{Me},  \cite{Ng2}). 
Our main objects of study are the sums
\begin{equation}
  \label{mainsum}
  S(\alpha,T,X,Y):=\sum_{0 < \gamma \leq T} \zeta(\rho+\alpha)X(\rho) Y(1\!-\! \rho)
\end{equation}  
\begin{equation}
  \label{SmTXY}
   S_m(T,X,Y):=   \sum_{0 < \gamma \leq T} \zeta^{(m)}(\rho)X(\rho) Y(1\!-\! \rho)
\end{equation}
where the factors $X(\cdot)$ and $Y(\cdot)$ are Dirichlet polynomials of length $N\geqslant 1$ defined by 
\begin{equation}
  X(s) = \sum_{n \le N} \frac{x(n)}{n^s} \text{ and }
  Y(s) = \sum_{n \le N} \frac{y(n)}{n^s}, 
\end{equation}
\begin{equation}\label{N vs T}
	N \ll T^{\vartheta} \text{ for some } \vartheta < \frac{1}{2}
\end{equation}
and where $\rho=\beta+i \gamma,\, 0<\sigma<1,\, \gamma\in\mathbb{R}$ ranges through the nontrivial zeros of the Riemann zeta function  $\zeta(s)$. 
Here $T \ge 2$ is a real parameter and $\alpha\in\mathbb{C}$ satisfies
\begin{align}\label{alpha size condition}
	\left|\alpha\right| \leqslant \frac{1}{15\log T}.
\end{align}

We state our main results in two forms: unconditional and conditional on a version of the Generalized Riemann Hypothesis (GRH) for Dirichlet $L$-functions in the following sense.

\begin{conj}[Conjecture $\text{GRH}(\Theta)$]  \label{GRHtheta}
There exists $\Theta  \in [\frac{1}{2},1)$ such that for all $q\geqslant 1$ and for all Dirichlet characters $\chi$ modulo $q$, the Dirichlet $L$-functions $L(s,\chi)$ have no zeros in the region $\sigma>\Theta$.
\end{conj}

We assume that $x(n)$ and $y(n)$ are supported on natural numbers $n\leqslant N$, that is, 
\begin{equation}
 \label{support}
x(n)=y(n)=0 \text{ for all } n>N
\end{equation}
and we assume 
a submultiplicativity condition that
\begin{equation}	\label{submultiplicative}
	x(mn) \ll |x(m)x(n)| \text{ and } y(mn) \ll |y(m)y(n)|
\end{equation}
for all  natural numbers $m$ and $n$.

Furthermore, for our unconditional results, we assume that 
 there exist
$k_1, k_2, \ell_1, \ell_2 \ge 1$ such that 
\begin{equation}	\label{xyinftyunc}
x(n) \ll \tau_{k_1}(n)(\log n)^{\ell_1} \text{ and } y(n) \ll  \tau_{k_2}(n)(\log n)^{\ell_2}
\end{equation}
where $\tau_k(\cdot)$ is the $k$-fold divisor function.

\begin{theorem} \label{mainthm}
Let $N \ge 1$ and $x,y$ be complex sequences satisfying the support condition \eqref{support} and the submultiplicativity condition \eqref{submultiplicative}. Let $\alpha \in \mathbb{C}$ satisfying \eqref{alpha size condition}. Let $s_{\alpha}(n) := n^{\alpha}$. Then we have
		\begin{align} \label{dmformula}
		\nonumber S(\alpha,T,X,Y) 
		&=\frac{T}{2\pi}\log\left(\frac{T}{2\pi e}\right) \sum_{n\leqslant N}\frac{\left(s_{-\alpha}\ast x\right) (n)y(n)}{n}
		-\frac{T}{2\pi}\sum_{n\leqslant N}\frac{\left(\Lambda\ast s_{-\alpha}\ast x\right) (n)y(n)}{n}
		\\
		& \quad \quad + \sum_{g\leqslant N}\sum_{\substack{h,k\leqslant N/g\\(h,k)=1}}\frac{y(gh)x(gk)}{gkh}
		\mathcal{F}_{\alpha,h,k}(T)  +\tilde{\mathcal{E}}
	\end{align}
 where $\mathcal{F}_{\alpha,h,k}(T)$ is defined in \eqref{definition_mathcal_F} below and where  we have the following bounds for the error term $\tilde{\mathcal{E}}$:
\begin{enumerate}[label=(\roman*)]
	\item If $x,y$ satisfy the bounds in \eqref{xyinftyunc}, then unconditionally, for any $A >0$
	\begin{align*}
		\tilde{\mathcal{E}} &\ll  
		T (\log T)^{-A}
	\end{align*}
for $T$ sufficiently large. 
\item Assuming the conjecture GRH($\Theta$), for every $\e >0$, we have
 \begin{align*}
    \tilde{\mathcal{E}}&\ll T^{\Theta+\varepsilon}  \left(\left\|\frac{y(n)}{n^\Theta}\right\|_1\left\|n^{1/2}x(n)(1\ast\left| y\right|)(n)\right\|_1 + \left\|\frac{x(n)y(n)}{n} \right\|_1 \left\|\frac{y(n)}{n^{\Theta}} \right\|_1 \left\|\frac{x(n)}{n} \right\|_1\left\|\frac{x(n)}{n^{2-\Theta}} \right\|_1 \right)
    \\&\quad \quad +T^{\frac{1}{2}+\varepsilon} \left(\left\| x \right\|_{1}\left\| \frac{y(n)}{n} \right\|_{1}+\left\| y \right\|_{1}\left\| \frac{x(n)}{n} \right\|_{1}\right)
\end{align*}
for $T$ sufficiently large. In particular, if $x(n)\ll n^{\delta_1}$ and $y(n)\ll n^{\delta_2}$  for some $\delta_1,\delta_2\geqslant 0$, then $\tilde{\mathcal{E}}\ll T^{\Theta+\varepsilon}N^{\frac{5}{2}-\Theta+\delta_1+2\delta_2}$.
\vspace{0.1cm}
 
\end{enumerate}
\end{theorem}
In Theorem \ref{mainthm}, we use the notation $\mathcal{F}_{\alpha,h,k}(T)$ for the function defined by  
 \begin{align}\label{definition_mathcal_F}
\mathcal{F}_{\alpha,h,k}(T):=\frac{T}{2\pi}\left(  	\frac{\mathbbm{1}_{k=1}}{h^{\alpha}}\frac{\zeta'}{\zeta}(1+\alpha)-\frac{\Lambda(k)}{h^{\alpha}\Phi(1+\alpha,k)}	
		-\frac{k}{\varphi(k)}\Phi(\alpha,k)\zeta(1-\alpha)\frac{\left(\frac{T}{2\pi k}\right)^{-\alpha}}{1-\alpha}\right)
\end{align}
where $\mathbbm{1}_{k=1}$ is the indicator function defined  by $\mathbbm{1}_{k=1}=1$ if $k=1$ and $\mathbbm{1}_{k=1}=0$ for $k>1$, $\Lambda(k)$ is the von Mangoldt function and $\varphi(k)$ is the Euler totient function and  for natural numbers $h$ and $k$, 
\begin{equation}
	\label{Phisk}
	\Phi(s,k):=\prod_{p\mid k}\left(1-p^{-s}\right), \quad  (s\in \mathbb{C}).
\end{equation}

\begin{remark}
    In Theorem \ref{mainthm}, in the conditional case, in order to establish an asymptotic formula for general Dirichlet coefficients, we must have $N<T^{\vartheta}$ with $\vartheta<\frac{1-\Theta}{5/2-\Theta+\delta_1+2\delta_2}$. When working with specific coefficients, the proof can be modified to work with larger $N$ and obtain better bounds for the error term. 
\end{remark}

 In \cite{Ng}, the third author  evaluated $S_1(T,X,Y)$ following the argument in \cite{CGG}.  Unfortunately, shortly  after the publication of \cite{Ng}, he spotted a mistake in the calculation of the error term in  \cite{CGG}  and the same mistake is carried in the calculation of the error term in \cite{Ng}. To be more precise, there   was  an improper application of the large sieve 
inequality in \cite{CGG}.  In \cite{BHB},  Bui and Heath-Brown corrected the argument in \cite{CGG} and even improved the claimed bounds there.  There were further minor errors in the main term of \cite[Theorem 1.3]{Ng}. These have been corrected in Theorem \ref{mainthm} in this work and the authors have now verified that the formula here agrees with a corrected version of the main term in \cite[Theorem 1.3]{Ng}.

In \cite{HLZ}, Heap, Li, and Zhao  estimated
\begin{equation}
  \label{2ndmoment}
   \sum_{0 < \gamma < T} \zeta(\rho+\alpha) \zeta(1-\rho+\beta) X(\rho) X(1-\rho) 
\end{equation}
by using the methods from \cite{BHB} by Bui and Heath-Brown. 
Note that $ \mathcal{F}_{\alpha,h,k}(T)$ also has the form 
\begin{equation}
  \label{F2ndformula}
  \mathcal{F}_{\alpha,h,k}(T) =  \overline{\frac{d}{d \gamma}  \frac{1}{2\pi}\int_{0}^{T}\left[Z_{\overline{\alpha},\gamma,h,k}+\left(\frac{t}{2\pi}\right)^{-\overline{\alpha}-\gamma}Z_{-\gamma,-\overline{\alpha},h,k}\right]\,dt \Bigr|_{\gamma=0}}
\end{equation}
where 
\begin{align}\label{definition of Z}
 Z_{\alpha,\gamma,h,k}:&=\frac{1}{h^{\alpha}k^{\gamma} }
\frac{\zeta(1+\alpha+\gamma)}{\zeta(1+\alpha)}
 \prod_{p \mid k}  
  \frac{    1-p^{\gamma} }{ 1 - p^{-1-\alpha}}
\end{align}
which matches with the expression obtained in \cite[Theorem 5]{HLZ}. A key difference in our work is that we show that the expression \eqref{F2ndformula} can further be  simplified to \eqref{definition_mathcal_F}.   In fact, by working with \eqref{definition_mathcal_F}, we are able to differentiate
the formula in (\ref{dmformula}) to obtain an asymptotic formula for  $S_m(T,X,Y)$.  Note that this is a more natural approach than several
previous computations of such sums.  
For instance, in \cite{CGG} and \cite{Ng} sums of the type $S_1(T,X,Y)$ are computed, but the authors work directly with the derivative $\zeta'(\rho)$ where the residue computations are significantly more complicated due to the fact that the coefficients involved are not multiplicative.  The idea of going from $\zeta(\rho+\alpha)$ to $\zeta^{(m)}(\rho)$ by differentiating seems to originate in the work of Ingham \cite{In}. Another difference between \cite{HLZ} and our work is that the coefficients $x(n)$ and $y(n)$ could be different. This allows us to establish a lower bound for the discrete $2k^{\text{th}}$ moments of the derivatives of the Riemann zeta function under the Riemann Hypothesis in Corollary \ref{corollary_lower_bound} below.  Moreover,  we obtain an error term under the assumption of the conjecture GRH($\Theta$) which is advantageous when working with specific Dirichlet coefficients and it even allows to work with larger coefficients for rather shorter Dirichlet polynomials.

\subsection{Definitions related to Corollary \ref{corollary}} In Corollary \ref{corollary} below, we estimate the $m^{\text{th}}$ derivative 
\begin{align*}
	\frac{d}{d\alpha^m}(S(\alpha,T,X,Y)) \Bigr|_{\alpha=0}= S_m(T,X,Y)=\sum_{0 < \gamma \leq T} \zeta^{(m)}(\rho)X(\rho) Y(1\!-\! \rho).
\end{align*}
Considering the definition of $\mathcal{F}_{\alpha,h,k}(T)$ in (\ref{definition_mathcal_F}), in order to state Corollary \ref{corollary}, we need the Taylor series of $\zeta(1-\alpha)$ and  $\zeta'/\zeta(1+\alpha)$ around $\alpha=0$. For this purpose, define the coefficients $\tilde{\gamma}_u$ and  $\eta_n$,  \cite{C}, by 
 \begin{align}\label{Laurent}
	\zeta(1-\alpha) = \sum_{u=-1}^{\infty}\tilde{\gamma}_u\alpha^u
  \text{ where }
 	\tilde{\gamma}_u = \begin{cases}
		-1 & \text{ if } u=-1, \\
		\frac{\gamma_u}{u!} & \text{ if } u \geqslant 0. 
	\end{cases}
\end{align}
 where the $\gamma_u$ are the Generalized Euler constants  and 
\begin{align} \label{Laurent2}
	\frac{\zeta'}{\zeta}(1+\alpha)=-\frac{1}{\alpha}-\sum_{n=0}^{\infty}\eta_{n}\alpha^{n}
\end{align}
where the coefficients $\eta_{n}$ satisfy the recurrence relation
\begin{align}\label{etanrecursion}
	\eta_{n}=-(-1)^n\left(\frac{(n+1)\gamma_n}{n!}+\sum_{k=0}^{n-1}\frac{(-1)^{k-1}}{(n-k-1)!}\eta_{k}\gamma_{n-k-1}\right)
\end{align}
by \cite[Equation A.7]{C}.

Throughout this article, we frequently encounter with the quantity
\begin{align} \label{Lo}
	 \mathscr{L} := \log \left(\frac{T}{2 \pi} \right).
\end{align}
We define the monic polynomials  $\mathcal{P}_{m+1}$ and $\mathcal{Q}_{m+1}$  of degree $m+1$ as
\begin{align}\label{definition_P}
	\mathcal{P}_{m+1}(x):=	x^{m+1}+(-1)^{m+1}(m+1)!\sum_{0\leqslant u \leqslant m}\frac{(-1)^{u}x^u}{u!}\left(1-\sum_{\substack{0\leqslant v\leqslant m-u}}\frac{\gamma_v}{v!}\right)
\end{align}
and
\begin{align}\label{definition_Q}
	\mathcal{Q}_{m+1}(x):=x^{m+1}+(-1)^{m+1}(m+1)!\sum_{0\leqslant u\leqslant m}\frac{(-1)^{u}\eta_{m-u}}{u!}x^u.
\end{align}
Let $\pl_{n}(z)$ be the polylogarithm function defined by
\begin{equation}
	\label{polylog}
	\pl_{n}(z):=\sum_{\ell=1}^{\infty}\frac{z^\ell}{\ell^n}, \quad \quad \left(\left|z\right|<1,\, n\in\mathbb{Z}\right). 
\end{equation}
For $j=0,1,2,\dots$, we have 
\begin{align}
	\label{polylogformula}
	\pl_{-j}(z)&= \sum_{\ell=1}^{\infty} \ell^j z^{\ell} = z^j \frac{d^j }{d z^j} \left(\frac{z}{1-z}\right)=\sum_{k=0}^{j} k! S(j+1,k+1) \left(\frac{z}{1-z}\right)^{k+1} 
\end{align}
where the $S(j,k)$ are Stirling numbers of the second kind. 

It turns out that the derivatives of $\mathcal{F}_{\alpha,h,k}(T)$ under consideration depends on the number of distinct prime factors $\omega(k)$ of $k$.
Let $h\geqslant 1$ and $k\geqslant2$ be natural numbers. If  $ k=p^a$ for some prime number $p$ and  $a\in\mathbb{N}$,  define
\begin{align}\label{definition_mathcal_A}
		\mathcal{A}_m(h,k)=\log^m h\log p + 
		\sum_{u_2=0}^m \binom{m}{u_2} (\log^{m-u_2} h)
		(\log^{u_2+1} p)\pl_{-u_2} \left(\frac{1}{p}\right),
\end{align}
and $\mathcal{A}_m(h,k)=0$ otherwise. 

Now, we define a function $\mathcal{B}_m(k,T)$ supported on numbers $k\geqslant 2$ with at most $m+1$ prime factors: For $k\in\mathbb{N}$, $k\geq 2$, if  $1 \le \omega(k) \le m+1$, define 

\begin{align}\label{definition_mathcal_B}
		\mathcal{B}_m(k,T):=m!\frac{k(-1)^{\omega(k)+1}}{\varphi(k)}
	\sum_{\substack{u_1+u_2\leqslant m\\u_1\geqslant \omega(k)-1\\u_2\geqslant 0}} \mathcal{G}_{u_1+1}(k)\frac{(-1)^{u_2}}{u_2!}\log^{u_2}\left(\frac{T}{2\pi k}\right)\sum_{j=-1}^{m-u_1-u_2-1}\tilde{\gamma}_{j},
\end{align}
where 
\begin{align}\label{definition_mathcal_G}
	\mathcal{G}_u(k) := (-1)^{u} \sum_{\substack{\ell_1+\ell_2+\dots+\ell_{\omega(k)}=u\\\ell_1,\dots, \ell_{\omega(k)}\geqslant 1}}\left(\frac{\log^{\ell_1}(p_1)}{\ell_1!}\frac{\log^{\ell_2}(p_2)}{\ell_2!}\dots\frac{\log^{\ell_{\omega(k)}}(p_{{\omega(k)}})}{\ell_{\omega(k)}!}\right),
\end{align}
and put  $\mathcal{B}_m(k,T):= 0$ if $\omega(k) \geqslant m+2$. Note that the largest exponent of $\mathscr{L}$ in $B_m(k,T)$ is $m$ which shows that $B_m(k,T)$ contributes to the secondary terms with respect to $T$ compared to the contribution that comes from the polynomial $\mathcal{P}_{m+1}(\mathcal{\Lo})$ of degree $m+1$.

\begin{corollary}\label{corollary}
Let $N \ge 1$ and $x,y$ be complex sequences satisfying the support condition \eqref{support} and the submultiplicativity condition \eqref{submultiplicative}. Then for $m\geqslant 1$, as $T\rightarrow\infty$, we have 
\begin{align}\label{mthderivative}
 & \nonumber\sum_{0 < \gamma \leq T} \zeta^{(m)}(\rho)X(\rho) Y(1\!-\! \rho) \nonumber=\frac{(-1)^{m+1}}{m+1}\frac{T}{2\pi }\sum_{g\leqslant N}\sum_{\substack{h\leqslant N/g}}\frac{y(gh)x(g)}{gh} \left(  \mathcal{P}_{m+1}(\Lo)  - \mathcal{Q}_{m+1}(\log h) \right)
      \\&\nonumber\quad +(-1)^m\frac{T}{2\pi}\log\left(\frac{T}{2\pi e}\right) 
      \sum_{n\leqslant N}\frac{\left( \log^m * x\right) (n)y(n)}{n}
      +(-1)^{m+1}\frac{T}{2\pi}\sum_{n\leqslant N}\frac{\left(\Lambda\ast  \log^m \ast x\right) (n)y(n)}{n}  \\
        &\quad+\frac{T}{2\pi }\sum_{g\leqslant N}\sum_{\substack{h,k\leqslant N/g\\k\geqslant2\\(h,k)=1}}\frac{y(gh)x(gk)}{gkh}\left((-1)^{m+1}\mathcal{A}_m(h,k)+\mathcal{B}_m(k,T)\right)
       +\tilde{\mathcal{E}}
\end{align} 
where $\mathcal{P}_{m+1}$ and $ \mathcal{Q}_{m+1}$ are
defined in \eqref{definition_P} and \eqref{definition_Q}, and $\mathcal{A}_m(h,k)$ and $\mathcal{B}_m(k,T)$ are defined in (\ref{definition_mathcal_A}) and (\ref{definition_mathcal_B}), respectively and $\tilde{\mathcal{E}}$ is as in Theorem \ref{mainthm}. 
\end{corollary}
Note that the first term  on the right-hand side of \eqref{mthderivative}  has the alternate form
\begin{align*}
	 \sum_{g\leqslant N}\sum_{\substack{h\leqslant N/g}}&\frac{y(gh)x(g)}{gh} \left(  \mathcal{P}_{m+1}(\Lo)  - \mathcal{Q}_{m+1}(\log h) \right) \\
	= & 
	\mathcal{P}_{m+1}(\Lo)   \sum_{n \leqslant N} \frac{   (1*x)(n) y(n)}{n}  -
	\sum_{n \leqslant N} \frac{ (\mathcal{Q}_{m+1}(\log)*x)(n)   y(n)}{n}.
\end{align*}
In the particular case when $m=1$, Corollary \ref{corollary} implies 
 \begin{equation}
 \begin{split}
 \label{firstderivative}
 	 \sum_{0 < \gamma \leq T} \zeta'(\rho)X(\rho) Y(1\!-\! \rho)& =\frac{T}{4\pi }  \mathcal{P}_{2}\left(\log\left(\frac{T}{2\pi}\right)\right) \sum_{n \leqslant N} \frac{   (1*x)(n) y(n)}{n}  -\frac{T}{4\pi }
 	\sum_{n \leqslant N} \frac{ (\mathcal{Q}_{2}(\log)*x)(n)   y(n)}{n} \\ 	&  -\frac{T}{2\pi}\log\left(\frac{T}{2\pi e}\right) 
 	\sum_{n\leqslant N}\frac{\left( \log * x\right) (n)y(n)}{n}
 	-\frac{T}{4\pi}\sum_{n\leqslant N}\frac{\left(\Lambda\ast  \log\ast x\right) (n)y(n)}{n}  \\
 	&+\frac{T}{2\pi }\sum_{g\leqslant N}\sum_{\substack{h,k\leqslant N/g\\k\geqslant2\\(h,k)=1}}\frac{y(gh)x(gk)}{gkh}\left(\mathcal{A}_1(h,k)+\mathcal{B}_1(k,T)\right)
 	+\tilde{\mathcal{E}}
 \end{split}
 \end{equation} 
where 
\begin{align*}
	\mathcal{P}_2(x)=x^2-2\left(1-\gamma_0\right)x+2\left( 1-\gamma_0-\gamma_1\right), \quad 
	\mathcal{Q}_2(x)  
	=x^2+2\gamma_0x+2\left(2\gamma_1+\gamma_0^2\right),
\end{align*}
\begin{align*}
	\small\mathcal{A}_1(h,k)
	=\begin{cases}
	 \frac{p\log h \log p}{p-1}+ \frac{p\log^2p}{(p-1)^2}    & \text{ if $k=p^a$, $p$ prime, $a\in \mathbb{N}$},   \\
		0 & \text{ otherwise}, 
	\end{cases}
\end{align*}
and 
\begin{align*}
	\small\mathcal{B}_1(k,T)= 
 \begin{cases}
 -\frac{p}{p-1}\left( \log p\left( \log\left(\frac{T}{2\pi}\right)-1+\gamma_0  \right)  + \left(a-\frac{1}{2}\right)\log^2p    \right)
 & \text{ if $k=p^a$, $p$ prime, $a\in \mathbb{N}$},   \\
  \frac{p_1p_2}{(p_1-1)(p_2-1)} \log p_1 \log p_2 & \text{ if $k=p_1^{a_1}p_2^{a_2}$, $p_1\neq p_2$ primes, $a_1,a_2\in\mathbb{N}$},\\
		0 & \text{ otherwise}.
 \end{cases}
\end{align*}
Moreover, for $m=1$, when  $N=1$ and $x(1)=y(1)=1$, our formula reduces to 
\begin{align*}
 \label{firstderivative and N=1}
 \sum_{0 < \gamma \leq T} \zeta'(\rho)=&\frac{T}{4\pi }  \mathcal{P}_{2}\left(\log\left(\frac{T}{2\pi}\right)\right) -\frac{T}{4\pi } \mathcal{Q}_{2}(0) +\mathcal{E}
 \\ =&\frac{T}{4\pi }  \log^2\left(\frac{T}{2\pi}\right)+\frac{T}{2\pi }\left(\gamma_0-1\right) \log\left(\frac{T}{2\pi}\right)+\frac{T}{2\pi } (1-\gamma_0-\gamma_0^2-3\gamma_1)+\mathcal{E},
 \end{align*} 
recovering the main term in \cite[Theorem 1]{F1} (see \cite[p. 52]{F2} for the corrected form). Moreover, for $m\geq 1$, the dominating main term in Corollary \ref{corollary} matches with the formula established by Kaptan, Karabulut, Y\i ld\i r\i m, \cite{KKY}. The coefficients of the lower order terms in Corollary \ref{corollary} match with the ones in \cite[Theorem 3]{HPC} in the work of Hughes and Pearce-Crump
as special cases by taking $N=1$. It also agrees with the work of Milinovich
 and the third author \cite{MN} and also with \cite{Ng} by the third author. 

The discrete $2k^{\text{th}}$ moment of the $m^{\text{th}}$ derivative of the Riemann zeta function is given by 
\begin{align*}
 \sum_{0<\gamma\leqslant T}\left|\zeta^{(m)}(\rho)\right|^{2k}, \quad (k\in\mathbb{R},\, \, m\in\mathbb{N}).
\end{align*}
If $k<0$, the moments are evaluated under the natural assumption that all the zeros of the Riemann zeta function are simple which is widely believed to be true. In \cite{G1}, Gonek obtained asymptotic formulas for $k=1$ and $m\in \mathbb{N}$
assuming the Riemann Hypothesis. For $k\neq 0,1$, no such asymptotic is known even conditionally.
Gonek \cite{G3} and Hejhal \cite{H} independently conjectured that
\begin{align*}
	T\left(\log T\right)^{k(k+2)+1}\ll \sum_{0<\gamma\leqslant T}\left|\zeta'(\rho)\right|^{2k} \ll T\left(\log T\right)^{k(k+2)+1}	
\end{align*}
for all $k\in\mathbb{R}$. The third author \cite{Ng0} established this conjecture in the case $k=2$, assuming the Riemann hypothesis.   Moreover, Hughes, Keating and O'Connell, \cite{HKO}, conjectured an asymptotic formula of order of the expected size above with a formulation of the coefficients using Random Matrix Theory, assuming $k >-3/2$. 

As an upper bound, Kirila,  \cite{K}, proved under the assumption of the Riemann Hypothesis that 
\begin{align*}
	\sum_{0<\gamma\leqslant T}\left|\zeta^{(m)}(\rho)\right|^{2k}\ll T\left(\log T\right)^{k(k+2m)+1} 
\end{align*}
for $k>0$ and $m\in \mathbb{N}$. Moreover,  in \cite{MN}, Milinovich and Ng proved  under the assumption of the Generalized Riemann Hypothesis that
\begin{equation}
 \label{lb}
\sum_{0<\gamma\leqslant T}\left|\zeta'(\rho)\right|^{2k}\gg T\left(\log T\right)^{k(k+2)+1}    
\end{equation}
for $k\in\mathbb{N}$. In \cite{G}, Gao generalizes this result to all $k \ge 0$, assuming the Riemann Hypothesis.  Gao references \cite{Ng} for one of his lemmas, but in order for his work to be valid he should employ the main theorem of this article instead.   There are also a number of results for negative moments.  Heap, Li, and Zhao \cite{HLZ} obtained a lower bound of the type \eqref{lb} for fractional $k \le 0$ and Gao and Zhao \cite{GZ} extended their work to real $k \le 0$.  There are very few upper bounds in the case $k \le 0$. There is a recent work by Bui, Florea, and Milinovich \cite{BFM} where an upper bounds for a related sum is obtained. 
In this work, as a result of Corollary \ref{corollary}, we generalize the lower bound \eqref{lb} for $k, m\in \mathbb{N}$ and we replace the assumption of the Generalized Riemann Hypothesis with the Riemann Hypothesis. We remark that it is possible to establish the same lower bound for all real $k>0$.
\begin{corollary}\label{corollary_lower_bound}
	Assume the Riemann Hypothesis. For all natural numbers $k,m\geqslant 1$, we have 
	\begin{align*}
		 \sum_{0 < \gamma \leqslant T} | \zeta^{(m)}(\rho)|^{2k} \gg T(\log T)^{k^2+2km+1}.
	\end{align*}
\end{corollary}

\subsection{Sketch of the proof.}
As the computation of  $S(\alpha,T,X,Y)$ is rather complicated we provide a brief sketch of the argument. 
By Cauchy's residue theorem and the functional equation of $\zeta(s)$, we have 
\[
  \mathcal{S} =S(\alpha,T,X,Y) = 
   -  \frac{1}{2 \pi i} \int_{\mathscr{C}} \frac{\zeta'}{\zeta}(1\!-\!s) \chi(s+\alpha) \zeta(1\!-\! s-\alpha) X(s) Y(1\!-\! s) \, ds
\]
where $\mathscr{C}$ is a positively oriented rectangle with vertices $1-\kappa+i,\kappa+i,\kappa+iT$ and $1-\kappa+iT$
where $\kappa:=1+\frac{1}{\log T}$. 
The horizontal contours contribute a negligible error and thus
$\mathcal{S} \sim \mathcal{S}_R+\mathcal{S}_L$
where
\begin{align*}
	\mathcal{S}_R&:=-  \frac{1}{2 \pi i} \int_{\kappa+i}^{\kappa+iT} \frac{\zeta'}{\zeta}(1\!-\!s)\zeta(s+\alpha) X(s) Y(1\!-\! s) \, ds, \\
	\mathcal{S}_L&:=-  \frac{1}{2 \pi i} \int_{1-\kappa+iT}^{1-\kappa+i} \frac{\zeta'}{\zeta}(1\!-\!s) \chi(s+\alpha) \zeta(1\!-\! s-\alpha) X(s) Y(1\!-\! s) \, ds.
\end{align*}
Note that in the first integral we use the identity $\chi(s+\alpha) \zeta(1-s-\alpha)=\zeta(s+\alpha)$.  The integral $\mathcal{S}_R$
can be evaluated by standard theorems on mean values of Dirichlet polynomials.  The hard part in the calculation is the evaluation of 
$\mathcal{S}_L$.     We make the variable change $w = 1-s-\alpha$ in $\mathcal{S}_L$ 
and then make use of the trick
\begin{equation}
  \label{HLZtrick}
  \frac{\zeta'}{\zeta}(w) = \frac{d}{d\gamma}   \frac{\zeta(w+\gamma)}{\zeta(w)} \Bigg|_{\gamma = 0} 
\end{equation}
to find that 

\begin{equation}
\begin{split}
  \label{SLvarchange}
			\mathcal{S}_L
			&= \overline{ \frac{d}{d \gamma} \frac{1}{2\pi i}\int_{\kappa+i-\overline{\alpha}}^{\kappa+iT-\overline{\alpha}}\frac{\zeta(w+\overline{\alpha}+\gamma)}{\zeta(w+\overline{\alpha})}\chi(1-w)\zeta(w)\overline{X}(1-w-\overline{\alpha})\overline{Y}(w+\overline{\alpha})\, dw
			\Big|_{\gamma=0}
			}.
	\end{split}
	\end{equation}
Note that the trick \eqref{HLZtrick}  of rewriting $\zeta'(w)/\zeta(w)$ is introduced in \cite{HLZ}.   This is a key step  as it  allows us to replace $\Lambda(n)$
by the multiplicative function $n \mapsto n^{-\gamma}$.  This simplifies the calculations of the main term and allows
us to make use of multiplicativity.
The computations in  \cite{CGG} and \cite{Ng} did not use this trick which resulted in that the calculations of their main terms are more complicated.   
One difference in our approach and that of Heap-Li-Zhao \cite{HLZ} is that we compute the derivatives $\frac{d}{d \gamma} \Big|_{\gamma=0}$ at the end of the argument, whereas their main term is in a form analogous to  \eqref{F2ndformula} (see \cite[p. 1592]{HLZ}). 
	
By absolute convergence of the Dirichlet series, all the sums in \eqref{SLvarchange} may be expanded out.  After doing this, it follows that $\mathcal{S}_L$ is a linear combination of integrals of the type $ \int_{c+i}^{c+iT} \chi(1-w)r^{-w}\,dw$ where $r=m/k \in \mathbb{Q}$. 
By Stirling's formula and the principle of stationary phase, it may be shown that 	
\begin{equation}
  \label{staphase}
	\frac{1}{2\pi i} \int_{c+i}^{c+iT} \chi(1-w)r^{-w}\,dw \sim  \mathbbm{1}_{[0,\frac{T}{2 \pi}]}(r) e(-r)
	\end{equation} 
where $r, c >0$ and $e(\theta) := e^{2 \pi i \theta}$. 
Precise versions of this result are well-known and may be found in 
Titchmarsh \cite[Lemma, p. 143, eq. (7.4.2), (7.4.3)]{T}, Levinson \cite[Lemmas 3.2, 3.2]{L}, Gonek \cite[Lemmas 1, 2, 3]{G}, 
and Conrey, Ghosh, and Gonek \cite[Lemma 1]{CGG}.  Applying \eqref{staphase} we find that 
\begin{align*}
	\mathcal{S}_L \sim 
	\overline{
	 \frac{d}{d \gamma}  \Bigg( \sum_{k\leqslant N}\frac{ \overline{x(k)k^{\alpha}} }{k}\sum_{m\leqslant kT/2\pi}a(m)e\left(-m/k\right)
	 \Bigg) \Big|_{\gamma=0}}
\end{align*}
for certain divisor-like coefficients $a(m)$. 

In order to deal with $e(-m/k)$  we make use of the identity 
\begin{align}\label{eq:emk}
	e(-m/k)
	= \frac{\mu(k')}{\phi(k')}
	+ \frac{1}{\phi(k')}  
	\sum_{{\begin{substack}{\chi \mod{k'}
					\\ \chi \ne \chi_0}\end{substack}}} 
	\tau(\overline{\chi}) \chi(-m'). 
\end{align}
where $m/k=m'/k'$ with $(m',k')=1$, 
where $\chi$ ranges over nonprincipal Dirichlet characters modulo $k'$, and 
$\tau(\chi)$ is the Gauss sum (for details see \cite[pp.121-122]{Ng}).  
Since $a(m)$ is defined by multiplicativity,  the contribution from $\mu(k')/\phi(k')$ can 
be computed by a straightforward application of Perron's formula as the generating Dirichlet series can 
be computed in terms of known functions. The coefficient $a(m) \approx (\mu*f_1*f_2*y)(m)$ where $f_1, f_2$ are bounded arithmetic functions.
The contribution from the second term in \eqref{eq:emk}
can be estimated by using
Heath-Brown's combinatorial decomposition for $\mu$ (see \cite{IK}) and then an application of the large sieve 
inequality.  

\subsection{A history of discrete mean values.} The discrete moments $\zeta'(\rho)$ that were  first studied in the 1980's are interesting in their own right, but they also have 
number theoretic consequences.  For instance, they can be used to deduce results on small and large gaps between zeros of $\zeta(s)$
and they have direct applications to the simple and multiple zeros of $\zeta(s)$. Furthermore, asymptotic formula for $S_1(T,X,Y)$ 
and \eqref{2ndmoment} have applications to obtaining lower bounds for high moments of $\zeta'(\rho)$ and $\zeta'(\rho)^{-1}$. 
 In 1984, Gonek \cite{G} proved the asymptotic formula  
\[
  \sum_{0 < \gamma < T} \zeta'(\rho) \zeta'(1-\rho) = \frac{T}{24 \pi} \log^4 T  +O(T \log^3 T). 
\]
More generally, Gonek showed  
$$ \sum_{0 < \gamma < T}  \zeta^{(j)}(\rho+\alpha)
\zeta^{(j')}(1-\rho-\alpha) =  C(j,j',\alpha) \frac{T}{2 \pi}  (\log T)^{j+j'+2}  + O(T (\log T)^{j+j'+1})
$$
where $C(j,j',\alpha)$ is holomorphic function of  $\alpha$  (see \cite[equation 4]{G}). Observe that  the Riemann hypothesis implies $ \zeta'(\rho) \zeta'(1-\rho) =|\zeta'(\tfrac{1}{2}+i \gamma)|^2$ which gives the  result of Gonek mentioned before Corollary \ref{corollary_lower_bound}. 
  In 1985, Conrey, Ghosh and Gonek \cite{CGG4} established 
\begin{equation}
  \label{firstmoment}
    \sum_{0 < \gamma < T} \zeta'(\rho)= \frac{T}{4 \pi} \log^2 T + O(T \log T)
\end{equation}
and in 1986 \cite{CGG1} they  established 
\begin{equation}
   \label{twistedmoment}
    \sum_{0 < \gamma < T} \zeta'(\rho) L(1-\rho, \chi) =   L(1,\chi) \frac{T}{4 \pi} \log^2 T + O(T \log T), 
\end{equation}
in the case that $\chi$ is a real primitive character modulo $|d|$. 
They used \eqref{firstmoment} to give a new proof that $\zeta(s)$ has infinitely many simple zeros and they used
\eqref{twistedmoment} to show that the Dedekind zeta function of quadratic number field has infinitely many simple zeros. 
In 1994,  Fujii \cite{F1} derived the asymptotic formula 
\begin{equation}
  \label{Fujii}
	\sum_{0<\gamma\leqslant T}\zeta(\rho+\alpha)=\frac{T}{2\pi }\log\left(\frac{T}{2\pi e}\right)+\frac{\zeta'}{\zeta}\left(1+\alpha \right)\frac{T}{2\pi }-\frac{\left(\frac{T}{2\pi}\right)^{1-\alpha}}{1-\alpha}\zeta(1-\alpha)+O\left(T\exp\left(-C\sqrt{\log T}\right)\right)	
\end{equation}
in the case $\alpha= i \Delta$ where $0\neq \Delta\ll 1$. 

In the special case $x(n)=y(n) =  \mu(n) P(\frac{\log N/n}{\log N})$
where $P$ is a polynomial
satisfying $P(0)=0$ and $P(1)=1$,  Conrey, Ghosh, and Gonek \cite{CGG} showed that the Generalized Lindel\"{o}f Hypothesis 
implies
\[
 \sum_{0 < \gamma < T}  \zeta'(\rho)X(\rho) \sim \lambda_1 \frac{T}{2 \pi} \Lo^3
\text{ and } 
   \sum_{0 < \gamma < T}  \zeta'(\rho)\zeta'(1-\rho) X(\rho)X(1-\rho) \sim \lambda_2 \frac{T}{2 \pi} \Lo^3
\]
for certain $\lambda_j=\lambda_j(P) \in \R$ (see \cite{CGG}). 
Conrey and Ghosh \cite{CG} initiated the study of discrete moments of degree two $L$-functions when they studied sums of the shape
\begin{equation}
  \label{discretemomentdeg2}
  \sum_{0 < \Im(\rho_f) < T} L'(\rho_{f}, f) F(\rho_{f}) 
\end{equation}
for $L(s,f)$ associated to holomorphic  newforms $f$, where $\rho_{f}$ ranges through the nontrivial zeros $L(s,f)$. 
They  showed that $L(s,f)$
has infinitely many simple zeros,  in the case $f=\Delta$ corresponds to Ramanujan's $\Delta$ function for an appropriate function $F(s)$.  In the case $F(s)=1$, it is an open problem to asymptotically evaluate such sums. 
In 2016, Booker \cite{B2} generalized Conrey-Ghosh's work and showed that any degree two modular $L$-function has infinitely many simple
zeros.  Other papers that  have built on this work and studied discrete moments for degree two $L$-function include \cite{BCK}, \cite{BMN}, \cite{MN3}, \cite{Fa}. Furthermore, ideas from Conrey-Ghosh's article \cite{CG} were used in Booker's work  \cite{B} on Artin's Holomorphy Conjecture.

\section{Conventions, Notation, and Preliminary Estimates} \label{conventions}
We use the convention that $\varepsilon$ denotes an arbitrarily small positive constant which may vary from line to line. 
The letter $p$ will always be used to denote a prime number.

Given two functions $f(x)$ and $g(x)$, we interchangeably use the notation  $f(x)=O(g(x))$, $f(x) \ll g(x)$, and $g(x) \gg f(x)$  to mean that there exists $M >0$ such that $|f(x)| \le M |g(x)|$ for all sufficiently large $x$. We write $f(x) \asymp g(x)$ to mean that the estimates $f(x) \ll g(x)$ and $g(x) \ll f(x)$ simultaneously hold. The constants implied in our big-$O$, $\ll$, and $\gg$ estimates are allowed to depend on $k$ and $\varepsilon$. 
For a sequence $\{z_n\}_{n\in \mathbb{N}}$ of complex numbers, we use the notation $$||z||_{1}:=
\sum_n|z_n|,$$
whenever exists. Given two arithmetic functions $f,g: \mathbb{N} \to \mathbb{C}$,  their Dirichlet convolution $f \ast g$ is defined by 
\[
  (f*g)(n) = \sum_{d_1d_2 =n} f(d_1) g(d_2). 
\]

In what follows, we provide the background information on the Riemann zeta-function and on arithmetic functions that we make use of in the sequel.

\subsection{The Riemann zeta-function}

The functional equation of $\zeta(s)$ is 
\begin{equation}  \label{eq:fe}
	\zeta(s) = \chi(s)\zeta(1\!-\!s) 
\end{equation}
where$\chi(s) = 2^s \pi^{s-1} \sin(\pi s/2) \Gamma(1\!-\!s)$.
Logarithmically differentiating the functional equation, it follows that
\begin{equation}  \label{eq:dfe}
   \frac{\zeta'}{\zeta}(1\!-\! s) = \frac{\chi'}{\chi}(s) - \frac{\zeta'}{\zeta}(s)
\end{equation}
where, by Stirling's formula for $\Gamma'(s)/\Gamma(s)$ and the identity $\chi(s)\chi(1-s)=1$, we have
\begin{equation}   \label{eq:chistirling}
  \frac{\chi'}{\chi}(\sigma\!+\!it) =   \frac{\chi'}{\chi}(1\!-\!\sigma \!-\!it) = - \log \Big( \frac{|t|}{2 \pi} \Big) + O(|t|^{-1})
\end{equation}
uniformly for $-1 \le \sigma \le 2$ and $|t| \ge 1$.  Moreover, for every $t\geq 2$ there exists a number $T$ satisfying $t\leq T\leq t+1$ such that 
\begin{equation} \label{eq:Tcondition}
  \frac{\zeta'}{\zeta}(\sigma \!+\! iT) \ll (\log T)^2 \quad \text{for } 
 -1 \le \sigma \le 2 \ \text{ and } \
 |\gamma\!-\! T| \gg (\log T)^{-1} 
\end{equation}
for all nontrivial zeros $\rho=\beta+i\gamma$ of $\zeta(s)$. This well-known argument may be found in \cite[p. 108]{Da}.

\subsection{Arithmetic functions}
The $k-$fold divisor function is  $\tau_{k}(n)=\underbrace{1\ast\dots\ast1}_{k \text{ times}}$ which can also be defined by the generating series
\begin{equation}\label{tau2}
\zeta^{k}(s) = \Big(\sum_{n=1}^{\infty}\frac{1}{n^{s}}\Big)^{k} = \sum_{n=1}^{\infty}\frac{\tau_{k}(n)}{n^{s}}
\end{equation}
for $\Re(s)> 1$. For $n\in \mathbb{N}$, $\omega(n)$ denotes the number of distinct prime factors of $n$. We have the M\"{o}bius function defined by
\begin{align*}
  \mu(n)=\begin{cases}
(-1)^{\omega(n)}\,\,&\text{if $n$ is squarefree, }\\
  0\,\,&\text{otherwise}
  \end{cases}
\end{align*}
 and  the von Mangoldt function defined by
 \begin{align*}
 \Lambda(n)=\begin{cases}
\log p\,\,\,&\text{if $n=p^a$  for some prime $p$ and  $a\in \mathbb{N}$,}\\
  0\,\,\,&\text{otherwise}
  \end{cases}
\end{align*}
that can also be given in terms of the generating series 
\[ 
\frac{1}{\zeta(s)}=\sum_{n=1}^\infty \frac{\mu(n)}{n^{s}} \quad \text{ and}\quad - \frac{\zeta'}{\zeta}(s)=\sum_{n=1}^\infty \frac{\Lambda(n)}{n^{s}}
\] 
for $\Re(s) >1$. 

\section{Lemmas}

We begin by quoting the following consequence by Tsang, \cite[Lemma 1]{Ts}, of Montgomery and Vaughan's mean value theorem for Dirichlet polynomials, \cite{MV}.
\begin{lemma} \label{Tsang lemma}
Let $\{a_n\}$ and $\{b_n\}$ be sequences of complex numbers.  For any real numbers $T$ and $H$, we have  
\begin{equation}\label{mvmvt1}
\int_{T}^{H}  \Bigg( \sum_{n=1}^{\infty} a_n n^{-it}  \Bigg) \Bigg( \overline{\sum_{n=1}^{\infty} b_n n^{-it}}  \Bigg) \ \!dt = H\sum_{n=1}^{\infty}a_n\overline{b_n} +O\left(   \bigg\{\sum_{n=1}^{\infty} n |a_n|^2\bigg\}^{\frac{1}{2}} \bigg\{ \sum_{n=1}^{\infty} n |b_n|^2 \bigg\}^{\frac{1}{2}} \right).
\end{equation}
\end{lemma}

By integrating by parts, we also have the following consequence of Tsang's lemma (see \cite[Lemma 4.1, p.3206]{MN}). 
\begin{lemma} \label{smoothmv}
Let $\{a_n\}$ and $\{b_n\}$ be sequences of complex numbers.  Assume that we have \\$\displaystyle{\sum_{n}n|a_n|^2<\infty}$ and $\displaystyle{\sum_{n}n|b_n|^2<\infty}$. Let $T_1$ and $T_2$ be positive real numbers and $g(t)$ be a real-valued function function which is continuously differentiable on the interval $[T_1,T_2]$. Then 
\begin{equation}\label{mvmvt}
\begin{split}
\int_{T_1}^{T_2}   g(t) \Bigg( \sum_{n=1}^{\infty} a_n n^{-it}  \Bigg) &\Bigg( \sum_{n=1}^{\infty} b_n n^{it}  \Bigg) \ \!dt = \left( \int_{T_1}^{T_2} g(t) \ \! dt  \right) 
\sum_{n=1}^\infty a_n b_n
\\
& \ + O\left( \bigg\{|g(T_2)|+\int_{T_1}^{T_2} |g'(t)| \ \! dt\bigg\}  \bigg\{\sum_{n=1}^{\infty} n |a_n|^2\bigg\}^{\frac{1}{2}} \bigg\{ \sum_{n=1}^{\infty} n |b_n|^2 \bigg\}^{\frac{1}{2}} \right).
\end{split}
\end{equation}
\end{lemma}

Our next lemma is a variant of Lemma 2 from \cite{CGG}. 
\begin{lemma}\label{CGG_type_lemma}
	Let $x(n)$ and $y(n)$ be sequences supported on $n\leqslant N\ll T^{\vartheta}$ for some $\vartheta <1/2$. Let $a(m)$ be a sequence such that $a(m)\ll m^{\frac{1}{5\log T}}(\tau_3\ast \left|y\right|)(m)$ and $b(k)=\overline{x(k)}k^{\overline{\alpha}}$ where $\left|\alpha\right|\leqslant \frac{1}{15\log T}$. Let $A(s)=\sum_{n=1}^{\infty}a(n)/n^{s}$ and $B(s)=\sum_{n\leqslant N}b(n)/n^{s}$ for $\sigma>1$.   Let $\kappa=1+\frac{1}{\log T}$. Then we have 
	\begin{align*}
		\frac{1}{2\pi i}\int_{\kappa+i-\overline{\alpha}}^{\kappa+iT-\overline{\alpha}}\chi(1-w)B(1-w)A(w)\, dw&=\sum_{k\leqslant N}\frac{b(k)}{k}\sum_{m\leqslant kT/2\pi}a(m)e\left(-m/k\right)
		\\&\quad \quad +O\left(  T^{1/2}\log^3 T\left\|x \right\|_1\left\|\frac{y(n)}{n} \right\|_1    \right).
	\end{align*}
\end{lemma}

\begin{proof}
		We closely follow the arguments in \cite[Lemma 2]{CGG} and \cite[Lemma 3.13]{Yunus}.	
		We have 
		\begin{align*}
			\frac{1}{2\pi i}\int_{\kappa+i-\overline{\alpha}}^{\kappa+iT-\overline{\alpha}}\chi(1-w)B(1-w)A(w)\, dw=\sum_{k\leqslant N}\frac{b(k)}{k}\sum_{m=1}^{\infty}a(m)\mathcal{J}(m,k,T)
		\end{align*}
		where 
		\begin{align*}
		\mathcal{J}(m,k,T):=	\frac{1}{2\pi i}\int_{\kappa+i-\overline{\alpha}}^{\kappa+iT-\overline{\alpha}}\chi(1-w)\left(\frac{m}{k}\right)^{-w}\, dw.	
		\end{align*}
	Since $\left|\alpha\right|$ is small enough, by Lemma 1 in \cite{CGG} (or see Lemma 2 in \cite{G1}), we have 
	\begin{align*}
		\mathcal{J}(m,k,T)=	\begin{cases}
			e(-m/k)+E(m/k, c)(m/k)^{-c} \,\,&\text{if}\, m/k\leqslant T/2\pi,\\
			E(m/k, c)(m/k)^{-c}\,\,&\text{otherwise},
		\end{cases}
	\end{align*}
	where $1+\frac{14}{15\log T}\leqslant c\leqslant  1+\frac{16}{15\log T}$ and 
	\begin{align*}
			E(m/k, c)\ll T^{c-1/2}+\frac{T^{c+1/2}}{\left|T- \frac{2\pi m}{k}\right|+T^{1/2}}\ll T^{1/2}+\frac{T^{3/2}}{\left|T- \frac{2\pi m}{k}\right|+T^{1/2}}.
	\end{align*}
The contribution of the term $e(-m/k)$ for $m\leqslant kT/2\pi$ gives the first term in the desired result. For the error term, it is sufficient to control 
\begin{align*}
	\sum_{k\leqslant N}\left|x(k)\right|\sum_{m=1}^{\infty}\frac{\left(\tau_3\ast \left|y\right|\right)(m)}{m^{1+\frac{11}{15\log T}}}\left( T^{1/2}+\frac{T^{3/2}}{\left|T-2\pi \frac{m}{k}\right|+T^{1/2}}  \right)
\end{align*}	
	since $(m/k)^{-c}\ll m^{-1-\frac{14}{15\log T}}k^{1+\frac{16}{15\log T}}\ll  m^{-1-\frac{14}{15\log T}}k$ as $k\leqslant N\ll T^{\vartheta}$. The contribution of the first term $T^{1/2}$ is bounded by 
	\begin{align*}
		\ll T^{1/2}\left\|x \right\|_1\sum_{d\leqslant N}\frac{\left|y(d)\right|}{d}\sum_{\ell=1}^{\infty}\frac{\tau_3(\ell)}{\ell^{1+\frac{11}{15\log T}}}\ll  T^{1/2}\log^3 T\left\|x \right\|_1\left\|\frac{y(n)}{n} \right\|_1.
	\end{align*}
	For the contribution of the second term, let 
	\begin{align*}
		f_{m,k}(T):=\frac{T^{3/2}}{\left|T- \frac{2\pi m}{k}\right|+T^{1/2}}.
	\end{align*}
Let $R:=\lfloor \frac{\log (T^{1/2})}{\log 2}-2\rfloor $. We consider the following five ranges
\begin{align*}
	&\frac{2\pi m}{k}\leqslant T-2^{R+1}T^{1/2}\leqslant  \frac{3T}{4},
	\\&T-2^{r+1}T^{1/2}<\frac{2\pi m}{k}\leqslant T-2^rT^{1/2},\quad 0\leqslant r\leqslant R,
	\\&T-T^{1/2}<\frac{2\pi m}{k}\leqslant T+T^{1/2},
\\&T+2^{r}T^{1/2}<\frac{2\pi m}{k}\leqslant T+2^{r+1}T^{1/2},\quad 0\leqslant r\leqslant R,
	\\&T+2^{R+1}T^{1/2}< \frac{2\pi m}{k}.
\end{align*}
In the first range, we have $f_{m,k}(T)\ll T^{1/2}$. Thus the contribution of this range is bounded by 
\begin{align*}
	\ll T^{1/2}\sum_{k\leqslant N}\left|x(k)\right|\sum_{m\leqslant \frac{kT}{4\pi}}\frac{\left(\tau_3\ast \left|y\right|\right)(m)}{m^{1+\frac{11}{15\log T}}}\ll T^{1/2}\log^3 T\left\|x \right\|_1\left\|\frac{y(n)}{n} \right\|_1.
\end{align*}
By the choice of $R$, in the second, third and the fourth ranges, we have $m\gg kT$. In the second and the fourth ranges we have  $f_{m,k}(T)\ll \frac{T^{3/2}}{2^rT^{1/2}}=\frac{T}{2^r}$ for $0\leqslant r\leqslant R$ and in the third range we have $f_{m,k}(T)\ll T$. Thus the total contribution of these three ranges is 
\begin{align*}
	\nonumber&\ll\sum _{k\leqslant N}\left|x(k)\right|\frac{1}{(kT)^{1+\frac{11}{15\log T}}}\sum_{d\leqslant N}\left|y(d)\right|\sum_{0\leqslant r\leqslant R}\frac{T}{2^r}\sum_{\substack{    \frac{k(T-2^{r+1}T^{1/2})}{2\pi d}<\ell\leqslant \frac{k(T-2^{r}T^{1/2})}{2\pi d}\\ \frac{k(T+2^{r}T^{1/2})}{2\pi d}<\ell\leqslant \frac{k(T+2^{r+1}T^{1/2})}{2\pi d} } }\tau_3(\ell)
	\\&\quad  +\sum _{k\leqslant N}\left|x(k)\right|\frac{1}{(kT)^{1+\frac{11}{15\log T}}}\sum_{d\leqslant N}\left|y(d)\right|\sum_{    \frac{k(T-T^{1/2})}{2\pi d}<\ell\leqslant \frac{k(T+T^{1/2})}{2\pi d} }T\tau_3(\ell)
	\\&\ll \sum _{k\leqslant N}\frac{\left|x(k)\right|}{k}\sum_{d\leqslant N}\left|y(d)\right|\sum_{0\leqslant r\leqslant R}\frac{1}{2^r}\sum_{\substack{    \frac{k(T-2^{r+1}T^{1/2})}{2\pi d}<\ell\leqslant \frac{k(T-2^{r}T^{1/2})}{2\pi d}\\ \frac{k(T+2^{r}T^{1/2})}{2\pi d}<\ell\leqslant \frac{k(T+2^{r+1}T^{1/2})}{2\pi d} } }\tau_3(\ell)
	\\&\quad  +\sum _{k\leqslant N}\frac{\left|x(k)\right|}{k}\sum_{d\leqslant N}\left|y(d)\right|\sum_{    \frac{k(T-T^{1/2})}{2\pi d}<\ell\leqslant \frac{k(T+T^{1/2})}{2\pi d} }\tau_3(\ell).
	\end{align*}
Note that for $d\leqslant N\ll T^{\vartheta}$ for some fixed $\vartheta<1/2$, we have $\frac{k2^rT^{1/2}}{d}\geqslant\frac{kT^{1/2}}{d} \gg T^{\frac{1}{2}-\vartheta}$ for all $0\leqslant r\leqslant R$. Thus we can apply Shiu's divisor sum bounds, \cite[Theorem 2]{S}, to the inner sums above which gives the upper bound
\begin{align*}
\ll T^{1/2}\log^2 T\sum _{k\leqslant N}\left|x(k)\right|\sum_{d\leqslant N}\frac{\left|y(d)\right|}{d}\sum_{0\leqslant r\leqslant R}1\ll T^{1/2}\log^3 T\left\|x \right\|_1\left\|\frac{y(n)}{n} \right\|_1.	
\end{align*}
Finally, in the fifth range, we have $f_{m,k}(T)\ll T^{1/2}$ by the choice of $R$ and the contribution of this range is 
\begin{align*}
	\ll T^{1/2}\sum_{k\leqslant N}\left|x(k)\right|\sum_{m\geqslant kT}\frac{\left(\tau_3\ast \left|y\right|\right)(m)}{m^{1+\frac{11}{15\log T}}}\ll T^{1/2}\log^3 T\left\|x \right\|_1\left\|\frac{y(n)}{n} \right\|_1 
\end{align*}
which finishes the proof of the desired result.

\end{proof}

Next, we quote Lemma 8 of \cite{HLZ}, where the second identity  is given in  Lemma 3 of \cite{CGG}.
\begin{lemma}\label{convolutiondecomposition}
Let $j,D\in \mathbb{N}$ and let $f_1, \dots , f_j$ be arithmetic functions. Given a decomposition of integers $D = d_1\dots d_j$, define $\displaystyle{D_i =\prod_{u=1}^{j-i}d_u}$ for $1\leqslant i\leqslant j-1 $ and $D_j=1$.
We have
\begin{align*}
	\sum_{\substack{m\leqslant x\\(m,k)=1}}(f_1\ast f_2\ast\dots\ast f_j)(mD)=\sum_{d_1\dots d_j=D}\sum_{\substack{m_1\dots m_j\leqslant x\\(m_i,kD_i)=1}}f_1(m_1d_j)f_2(m_2d_{j-1})\dots f_j(m_jd_1),
\end{align*}
and 
\begin{align*}
\sum_{\substack{m=1\\(m,k)=1}}^{\infty}\frac{(f_1\ast f_2\ast\dots\ast f_j)(mD)}{m^s}=	\sum_{d_1\dots d_j=D}\prod_{i=1}^{j}\sum_{\substack{m_i=1\\(m_i,kD_i)=1}}^{\infty}\frac{f_i(m_{i}d_{j+1-i})}{m_i^{s}}.
\end{align*}

\end{lemma}

To estimate partial sums of the coefficients of some Dirichlet series, we use the following version of Perron's formula. 

\begin{lemma}[Perron's formula, Theorem 2.1 of \cite{LY}]
\label{perron}
Let $f(s) =\sum_{n=1}^\infty \frac{a_n}{n^s}$ be a Dirichlet series with abscissa of absolute convergence $\sigma_a$. Let
$B(\sigma) =\sum_{n=1}^\infty \frac{|a_n|}{n^\sigma}$
for $\sigma>\sigma_a$. Then for $\kappa> \sigma_a$, $x \geqslant2$, $U\geqslant 2$, and $H\geqslant 2$, we have
\begin{align*}
	\sum_{n\leqslant x} a_n=\frac{1}{2\pi i}\int_{\kappa-iU}^{\kappa+iU}f(s)\frac{x^s}{s}\,ds +O\left(\sum_{x-\frac{x}{H}\leqslant n\leqslant x+\frac{x}{H}} \left|a_n\right|\right)+O\left(\frac{x^\kappa HB(\kappa)}{U}\right).
\end{align*}
\end{lemma}

In the following lemmas, we calculate the derivatives of functions appearing in our main terms with respect to  the variables $\alpha$ and $\gamma$ at $0$.

\begin{lemma}\label{Lemma_for_Z_gamma}
	Let  $Z_{\alpha,\gamma,h,k}$ be defined by (\ref{definition of Z}) and $\alpha\neq 0$. Then 
	\begin{align*}
		 \frac{d}{d \gamma} Z_{\alpha,\gamma,h,k}  \Bigr|_{\gamma=0}=\begin{cases}
		 	\frac{1}{h^{\alpha}}\frac{\zeta'}{\zeta}(1+\alpha) & \text{if}\,\, k=1,
		 \\-\frac{\Lambda(k)}{h^{\alpha}\Phi(1+\alpha,k)}	&\text{if}\,\, k>1.
		 \end{cases}	
	\end{align*}
where $\Phi(s,k)$ is defined in \eqref{Phisk}. 
Moreover, we have 
\begin{align*}
	 \frac{d}{d \gamma} Z_{-\gamma,\alpha,h,k}  \Bigr|_{\gamma=0}=-\frac{k}{\varphi(k)}\frac{\Phi(-\alpha,k)\zeta(1+\alpha)}{k^{\alpha}}.
\end{align*}
\end{lemma}
\begin{proof}
 We have 
\begin{equation}\label{Z_using_a_lemma}
	 Z_{\alpha,\gamma,h,k}=\frac{1}{h^{\alpha}}
	\frac{\zeta(1+\alpha+\gamma)}{\zeta(1+\alpha)}\frac{1}{k^{ \gamma}}  \prod_{p \mid k}  
	\frac{    1-p^{\gamma} }{ 1 - p^{-1-\alpha}}
	=\frac{1}{h^{\alpha}}
	\frac{\zeta(1+\alpha+\gamma)}{\zeta(1+\alpha)}\frac{1}{k^{ \gamma}} \frac{\Phi(-\gamma,k)}{\Phi(1+\alpha,k)}.
\end{equation}
Observe that if $k=1$, then the empty product is interpreted as $1$, and in this case, we have
\begin{align*}
	 \frac{d}{d \gamma} Z_{\alpha,\gamma,h,1}  \Bigr|_{\gamma=0}= \frac{d}{d \gamma}  \frac{1}{h^{\alpha}}
	 \frac{\zeta(1+\alpha+\gamma)}{\zeta(1+\alpha)} \Bigr|_{\gamma=0}=\frac{1}{h^{\alpha}}\frac{\zeta'}{\zeta}(1+\alpha).
\end{align*}	
	If $k>1$, we have 
	\begin{align*}
	 \frac{d}{d \gamma} Z_{\alpha,\gamma,h,k}  \Bigr|_{\gamma=0}&=\frac{1}{h^{\alpha}\Phi(1+\alpha,k)\zeta(1+\alpha)}\left(\zeta'(1+\alpha)\Phi(0,k)+\zeta(1+\alpha)\frac{d}{d\gamma}\left(k^{-\gamma}\Phi(-\gamma,k)\right)\Bigr|_{\gamma=0}\right)	
	 \\&=\frac{1}{h^{\alpha}\Phi(1+\alpha,k)}\frac{d}{d\gamma}\left(k^{-\gamma}\Phi(-\gamma,k)\right)\Bigr|_{\gamma=0}
	\end{align*}
	since $\Phi(0,k)=0$ as $k>1$. Observe that if $k$ has more than one distinct prime factors, then
	\begin{align*}
		\frac{d}{d\gamma}\Phi(-\gamma,k)\Bigr|_{\gamma=0}=0
	\end{align*}
since at least one of the factors will be zero in each term after differentiation and evaluation at $\gamma=0$. If $k$ is a prime power, $k=p^m$ for some $m\geqslant1$ say, then
\begin{align*}
	\frac{d}{d\gamma}\Phi(-\gamma,k)\Bigr|_{\gamma=0}=	\frac{d}{d\gamma}(1-p^{\gamma})\Bigr|_{\gamma=0}=-\log p.
\end{align*}
Therefore, if $k>1$, we have  
$\frac{d}{d\gamma}\Phi(-\gamma,k)\Bigr|_{\gamma=0}=-\Lambda(k)$
	and thus 
	\begin{align*}
		\frac{d}{d\gamma}\left(k^{-\gamma}\Phi(-\gamma,k)\right)\Bigr|_{\gamma=0}=-\Lambda(k)
	\end{align*}
since $\Phi(0,k)=0$ for $k>1$. Hence,
\begin{align*}
\frac{d}{d \gamma} Z_{\alpha,\gamma,h,k}  \Bigr|_{\gamma=0}=\begin{cases}
\frac{1}{h^{\alpha}}\frac{\zeta'}{\zeta}(1+\alpha) & \text{if}\,\, k=1,
\\-\frac{\Lambda(k)}{h^{\alpha}\Phi(1+\alpha,k)}	&\text{if}\,\, k>1
\end{cases}	
\end{align*}
which finishes the proof of the first assertion. For the second part, we have 
\begin{align}\label{Second_part_Z}
	Z_{-\gamma,\alpha,h,k}
	=\frac{1}{h^{-\gamma}}
	\frac{\zeta(1-\gamma+\alpha)}{\zeta(1-\gamma)}\frac{1}{k^{\alpha}} \frac{\Phi(-\alpha,k)}{\Phi(1-\gamma,k)}
\end{align}
by (\ref{Z_using_a_lemma}). Thus
\begin{align*}
	 \frac{d}{d \gamma} Z_{-\gamma,\alpha,h,k}  \Bigr|_{\gamma=0}&=\frac{\Phi(-\alpha,k)}{k^{\alpha}}\frac{d}{d \gamma}\left(\frac{h^{\gamma}\zeta(1-\gamma+\alpha)}{\Phi(1-\gamma,k)}\frac{1}{\zeta(1-\gamma)}\right)  \Bigr|_{\gamma=0}
	 \\&=\frac{\Phi(-\alpha,k)\zeta(1+\alpha)}{k^{\alpha}\Phi(1,k)}\frac{d}{d \gamma} \frac{1}{\zeta(1-\gamma)}  \Bigr|_{\gamma=0}
\end{align*}
since $\alpha\neq 0$ and $\lim_{\gamma\rightarrow 0}\left|\zeta(1-\gamma)\right| =\infty$. Since 
\begin{align*}
\frac{d}{d \gamma} \frac{1}{\zeta(1-\gamma)}  \Bigr|_{\gamma=0}=\lim_{\gamma\rightarrow 0}\frac{\zeta'(1-\gamma)}{\zeta^2(1-\gamma)}=\lim_{\gamma\rightarrow 0}\frac{-\frac{1}{(1-\gamma-1)^2}}{\zeta^2(1-\gamma)}=-\lim_{\gamma\rightarrow 0}\frac{1}{\gamma^2\zeta^2(1-\gamma)}=-1,
\end{align*} 
we have 
\begin{align*}
 \frac{d}{d \gamma} Z_{-\gamma,\alpha,h,k}  \Bigr|_{\gamma=0}=	-\frac{\Phi(-\alpha,k)\zeta(1+\alpha)}{k^{\alpha}\Phi(1,k)}=-\frac{k}{\varphi(k)}\frac{\Phi(-\alpha,k)\zeta(1+\alpha)}{k^{\alpha}}
\end{align*}
which finishes the proof the lemma.
\end{proof}	

\begin{lemma}\label{lemma_simplification_of_F}
	Let $\alpha\in\mathbb{C}\setminus\lbrace 0,1\rbrace$. Let   $Z_{\alpha,\gamma,h,k}$ be defined by (\ref{definition of Z}). Recall that  $\mathcal{F}_{\alpha,h,k}(T)$ is defined in (\ref{definition_mathcal_F}) by 
	\begin{align*}
		\mathcal{F}_{\alpha,h,k}(T)=\frac{T}{2\pi}\left(  	\frac{\mathbbm{1}_{k=1}}{h^{\alpha}}\frac{\zeta'}{\zeta}(1+\alpha)-\frac{\Lambda(k)}{h^{\alpha}\Phi(1+\alpha,k)}	
		-\frac{k}{\varphi(k)}\Phi(\alpha,k)\zeta(1-\alpha)\frac{\left(\frac{T}{2\pi k}\right)^{-\alpha}}{1-\alpha}\right).
	\end{align*}
	 We have 
		\begin{align}\label{defnCurlyF}
		\overline{\frac{d}{d \gamma}  \frac{1}{2\pi}\int_{0}^{T}\left(Z_{\overline{\alpha},\gamma,h,k}+\left(\frac{t}{2\pi}\right)^{-\overline{\alpha}-\gamma}Z_{-\gamma,-\overline{\alpha},h,k}\right)\,dt \Biggr|_{\gamma=0}}=	\mathcal{F}_{\alpha,h,k}(T).
	\end{align}

\end{lemma}
\begin{proof}
By the second assertion in Lemma \ref{Lemma_for_Z_gamma} and the fact that $Z_{0,-\overline{\alpha},h,k}=0$ by (\ref{Second_part_Z}), we have 
\begin{align*}
	\frac{d}{d \gamma}\left(\left(\frac{t}{2\pi}\right)^{-\overline{\alpha}-\gamma}Z_{-\gamma,-\overline{\alpha},h,k}\right)\Biggr|_{\gamma=0}=-\frac{k}{\varphi(k)}\left(\frac{t}{2\pi}\right)^{-\overline{\alpha}}\frac{\Phi(\overline{\alpha},k)\zeta(1-\overline{\alpha})}{k^{-\overline{\alpha}}}.
\end{align*}
Thus, by Lemma \ref{Lemma_for_Z_gamma}, we obtain the desired result after integration and conjugation.

\end{proof}

Now, we start considering the derivative $\frac{d^m}{d\alpha^m}	\mathcal{F}_{\alpha,h,k}(T) \Bigr|_{\alpha=0}$ which is used in the proof of Corollary \ref{corollary}. We start with the relatively simpler case when $k=1$ for which  the $\Phi(\cdot,\cdot)$ factors are taken to be $1$.

\begin{lemma}\label{lemma_derivative_1}
	Let $\mathcal{F}_{\alpha,h,k}(T)$ be defined by (\ref{definition_mathcal_F}) and let the monic polynomials $\mathcal{Q}_{m+1}(\cdot)$ and $\mathcal{P}_{m+1}(\cdot)$ of degree $m+1$ be defined by (\ref{definition_P}) and (\ref{definition_Q}). For $k=1$ and  $m\geqslant 0$, we have 	
	\begin{align}\label{derivative_k=1_final}
		\frac{d^m}{d\alpha^m}	\mathcal{F}_{\alpha,h,1}(T) \Biggr|_{\alpha=0}
		=\frac{T}{2 \pi}  \frac{(-1)^{m+1}}{m+1} \left(  \mathcal{P}_{m+1}(\Lo)  - \mathcal{Q}_{m+1}(\log h) \right).
	\end{align}	
In particular, if $m=1$, then we have 
\begin{align}\label{particular_m=1_k=1}
		\frac{d}{d\alpha}	\mathcal{F}_{\alpha,h,1}(T) \Biggr|_{\alpha=0}=\frac{T}{2\pi}\left(\frac{\Lo^2}{2}+(\gamma_0-1)\Lo-\frac{\log^2 h}{2}-\gamma_{0}\log h +1-\gamma_0-\gamma_{0}^2-3\gamma_1\right).
\end{align}

\end{lemma}
\begin{proof}
	Let $\left|\alpha\right|<1$. By Lemma \ref{lemma_simplification_of_F} and the Taylor series given  in (\ref{Laurent}) and (\ref{Laurent2}), we have
	\begin{align*}
		\nonumber \mathcal{F}_{\alpha,h,1}(T)&=\frac{T}{2\pi}\frac{1}{h^{\alpha}}\frac{\zeta'}{\zeta}(1+\alpha)-\zeta(1-\alpha)\frac{\left(\frac{T}{2\pi}\right)^{1-\alpha}}{1-\alpha}
		\\&\nonumber=\frac{T}{2\pi h^\alpha}\left(-\frac{1}{\alpha}-\sum_{\ell=0}^{\infty}\eta_{\ell}\alpha^{\ell}\right)-\frac{\left(\frac{T}{2\pi}\right)^{1-\alpha}}{1-\alpha}\left(-\frac{1}{\alpha}+\sum_{\ell=0}^{\infty}\frac{\gamma_{\ell}}{\ell!}\alpha^{\ell}\right)\\
		& =  -\frac{T}{2 \pi} 
	\left(
		\frac{1}{\alpha} \Big( h^{-\alpha} -\frac{\left(\frac{T}{2\pi}\right)^{-\alpha}}{1-\alpha} \Big)
		+ h^{-\alpha} \sum_{\ell=0}^{\infty} \eta_{\ell} \alpha^{\ell}
		+ \frac{\left(\frac{T}{2\pi}\right)^{-\alpha}}{1-\alpha}  \sum_{\ell=0}^{\infty}\frac{\gamma_{\ell}}{\ell!} \alpha^{\ell} \right).
	\end{align*}
Note that in the Taylor series $\frac{1}{\alpha} \Big( h^{-\alpha} -\frac{\left(\frac{T}{2\pi}\right)^{-\alpha}}{1-\alpha} \Big)$, the  term with $u=0$ vanishes. By writing
 \begin{align*}
 	h^{-\alpha} \sum_{\ell=0}^{\infty} \eta_{\ell} \alpha^{\ell}=\sum_{u=0}^{\infty}\sum_{\substack{u_1+u_2=u\\u_1,u_2\geqslant 0}} \frac{(-1)^{u_1} \log^{u_1} h}{u_1!} \eta_{u_2}\alpha^u
 \end{align*}
and 
\begin{align*}
 \frac{\left(\frac{T}{2\pi}\right)^{-\alpha}}{1-\alpha}  \sum_{\ell=0}^{\infty}\frac{\gamma_{\ell}}{\ell!} \alpha^{\ell}=\sum_{u=0}^{\infty} \sum_{\substack{u_1+u_2 + u_3=u\\u_1,u_2,u_3\geqslant 0}} \frac{(-1)^{u_1} \Lo^{u_1} }{u_1!} \frac{\gamma_{u_2}}{u_2!}\alpha^u,
\end{align*}
 we have 
 \begin{align*}
 \mathcal{F}_{\alpha,h,1}(T)&=-\frac{T}{2 \pi}\sum_{u=0}^{\infty}   \left(      	 \frac{(-1)^{u+1} \log^{u+1} h}{(u+1)!} -   \sum_{\substack{u_1+u_2=u+1\\u_1,u_2\geqslant 0}}
 \frac{(-1)^{u_1} \Lo^{u_1}}{u_1!} \right. \\&\quad \quad \left.      \quad \quad +\sum_{\substack{u_1+u_2=u\\u_1,u_2\geqslant 0}} \frac{(-1)^{u_1} \log^{u_1} h}{u_1!} \eta_{u_2}         +       \sum_{\substack{u_1+u_2 + u_3=u\\u_1,u_2,u_3\geqslant 0}} \frac{(-1)^{u_1} \Lo^{u_1} }{u_1!} \frac{\gamma_{u_2}}{u_2!}                            \right)\alpha^{u}. 	
 \end{align*}
Thus, for $m\geqslant 0$, we have
\begin{align}\label{derivative_k=1}
	\nonumber \frac{d^{m}}{d\alpha^m}  \mathcal{F}_{\alpha,h,1}(T) \Biggr|_{\alpha=0}&=\frac{T}{2\pi}\left( m!   \sum_{\substack{u_1+u_2=m+1\\u_1,u_2\geqslant 0}}
	\frac{(-1)^{u_1} \Lo^{u_1}}{u_1!}  -      m! \sum_{\substack{u_1+u_2 + u_3=m\\u_1,u_2,u_3\geqslant 0}} \frac{(-1)^{u_1} \Lo^{u_1} }{u_1!} \frac{\gamma_{u_2}}{u_2!}           \right. \\&\quad \quad \left.      \quad \quad      -	 m!\frac{(-1)^{m+1} \log^{m+1} h}{(m+1)!} -m!\sum_{\substack{u_1+u_2=m\\u_1,u_2\geqslant 0}} \frac{(-1)^{u_1} \log^{u_1} h}{u_1!} \eta_{u_2}                         \right).
\end{align}

For the first two terms in the parenthesis above, we have 
\begin{align}\label{first_two_terms}
 \nonumber&m!   \sum_{\substack{u_1+u_2=m+1\\u_1,u_2\geqslant 0}}
\frac{(-1)^{u_1} \Lo^{u_1}}{u_1!}  -      m! \sum_{\substack{u_1+u_2 + u_3=m\\u_1,u_2,u_3\geqslant 0}} \frac{(-1)^{u_1} \Lo^{u_1} }{u_1!} \frac{\gamma_{u_2}}{u_2!}     
\\&\nonumber=\frac{(-1)^{m+1}}{m+1}\Lo^{m+1}	+m!\left(\sum_{\substack{u_1+u_2=m+1\\u_1,u_2\geqslant 0\\u_1\leqslant m}}
\frac{(-1)^{u_1} \Lo^{u_1}}{u_1!}- \sum_{\substack{u_1+u_2 + u_3=m\\u_1,u_2,u_3\geqslant 0}} \frac{(-1)^{u_1} \Lo^{u_1} }{u_1!} \frac{\gamma_{u_2}}{u_2!}     \right)
\\&=\frac{(-1)^{m+1}}{m+1}\mathcal{P}_{m+1}(\Lo)
\end{align}
where $\mathcal{P}_{m+1}$ is defined in \eqref{definition_P}.
Now, we consider the terms with powers of $\log h$. We have 
\begin{align}\label{last_two_terms}
	-	 m!\frac{(-1)^{m+1} \log^{m+1} h}{(m+1)!} -m!\sum_{\substack{u_1+u_2=m\\u_1,u_2\geqslant 0}} \frac{(-1)^{u_1} \log^{u_1} h}{u_1!} \eta_{u_2}
	=-\frac{(-1)^{m+1}}{m+1}\mathcal{Q}_{m+1}(\log h) 
\end{align}
where $\mathcal{Q}_{m+1}$ is defined in \eqref{definition_Q}. 

 Hence, we obtain the first assertion (\ref{derivative_k=1_final}) by (\ref{derivative_k=1})-(\ref{last_two_terms}). For the particular case where $m=1$, we have 
\begin{align*}
		\frac{d}{d\alpha}	\mathcal{F}_{\alpha,h,1}(T) \Biggr|_{\alpha=0}&=\frac{T}{2\pi}\frac{1}{2}   \left[\Lo^2+2\left(-\Lo (1-\gamma_0)+1-\gamma_0-\gamma_1\right) 
		 -\log^2 h-2\left( -\eta_0 \log h+\eta_1    \right)           \right]
		 \\&=\frac{T}{2\pi}\left(\frac{\Lo^2}{2}+(\gamma_0-1)\Lo-\frac{\log^2 h}{2}+\eta_0\log h +1-\gamma_0-\gamma_1-\eta_1\right).    
\end{align*}
By the recursive formula (\ref{etanrecursion}), we have $\eta_0=-\gamma_0$ and $\eta_1=2\gamma_1+\gamma_{0}^2$.
Hence, we obtain (\ref{particular_m=1_k=1}) which finishes the proof.

\end{proof}

We continue with the derivative $\frac{d^m}{d\alpha^m}	\mathcal{F}_{\alpha,h,k}(T) \Bigr|_{\alpha=0}$ when $k>1$. We start with the Taylor series for the terms involving the function $\Phi(\cdot,\cdot)$.
\begin{lemma}  \label{powerseries}
	Let  $\alpha\in\mathbb{C}$ with $\left|\alpha\right|<1$.
\begin{enumerate}[label=(\roman*)]
\item  Let $p$ be a prime number and $a\in\mathbb{N}$. We have
\begin{align}\label{ps1}
	\frac{1}{\Phi(1+\alpha,p^a)} = \sum_{u=0}^{\infty}   \left(\mathbbm{1}_{ u=0}+(-1)^{u}  (\log^u p)\emph\pl_{-u}\left(\frac{1}{p}\right) \right) \frac{\alpha^u}{u!}
\end{align}
where the polylogarithm function $\emph\pl_{-u}\left(\frac{1}{p}\right)$ is defined by (\ref{polylog}).

\item Let $k\geqslant 2$ and $\omega(k)$ be the number of distinct prime factors of $k$. Let  $k=\prod_{j=1}^{\omega(k)}p_j^{a_j}$ be the unique prime factorization of $k$. We have 
\begin{align}\label{ps2}
		\Phi(\alpha,k)= (-1)^{\omega(k)}
		\sum_{u={\omega(k)}}^{\infty}\mathcal{G}_u(k)  \alpha^{u} 
\end{align}
where $\mathcal{G}_u(k)$ is defined in \eqref{definition_mathcal_G}.

\end{enumerate}

\end{lemma}

\begin{proof}  
\begin{enumerate}[label=(\roman*)]
    \item 

	 We have 
	\begin{align*}
	 \Phi(1+\alpha,p^{a})^{-1}=
		\left(1-p^{-1-\alpha}\right)^{-1}=1+\frac{1}{p^{1+\alpha}-1}=1+\sum_{\ell=0}^{\infty}\frac{1}{p^{\left(1+\alpha\right)(\ell+1)}}.
	\end{align*}
Thus the $u^{\text{th}}$ coefficient of the Taylor expansion of  $\Phi(1+\alpha,p^{a})^{-1}$ about $\alpha=0$ is given by
	\begin{align*} \frac{1}{u!}\frac{d^u}{d \alpha^u} 
	\Phi(1+\alpha,p^{a})^{-1} \Biggr|_{\alpha=0} &=\frac{1}{u!}\mathbbm{1}_{ u=0}+\frac{1}{u!}\sum_{\ell=0}^{\infty}\frac{1}{p^{\ell+1}}\frac{d^u}{d\alpha^u}\left(p^{-\alpha(\ell+1)}\right)\Biggr|_{\alpha=0}
\\&=\mathbbm{1}_{ u=0}+\frac{1}{u!}(-1)^{u}  \log^u{p}\,\pl_{-u}\left(\frac{1}{p}\right)
	\end{align*}
by the definition of the polylogarithm function $\pl_n(z)$ in \eqref{polylog}. This finishes the proof of the first assertion.  
\item For the second assertion, let  $k=\prod_{j=1}^{\omega(k)}p_j^{a_j}$ be the unique prime factorization of $k>1$. We have 
\begin{align*}
	\Phi(\alpha,k) & = \prod_{p\mid k}\left(1-p^{-\alpha}\right)=
	\prod_{j=1}^{\omega(k)}\left(-\sum_{\ell=1}^{\infty}\frac{(-1)^{\ell}\alpha^{\ell} \log^{\ell}(p_j)}{\ell!}\right) \\
	& = (-1)^{\omega(k)}\sum_{u=\omega(k)}^{\infty}(-1)^{u}\sum_{\substack{\ell_1+\ell_2+\dots+\ell_{\omega(k)}=u\\\ell_1,\dots, \ell_{\omega(k)}\geqslant 1}}\left(\frac{\log^{\ell_1}(p_1)}{\ell_1!}\frac{\log^{\ell_2}(p_2)}{\ell_2!}\dots\frac{\log^{\ell_{\omega(k)}}(p_{{\omega(k)}})}{\ell_{\omega(k)}!}\right)\alpha^{u},
\end{align*}
by multiplying out the products of the series. Note that the coefficient of $\alpha^u$ is $\mathcal{G}_u(k) $ defined in \eqref{definition_mathcal_G} and the proof is complete. 
\end{enumerate}
\end{proof}

\begin{lemma}\label{lemma_F_for_k>1}
 Recall that  $\mathcal{F}_{\alpha,h,k}(T)$ is defined in (\ref{definition_mathcal_F}) by 
\begin{align*}
	\mathcal{F}_{\alpha,h,k}(T)=\frac{T}{2\pi}\left(  	\frac{\mathbbm{1}_{k=1}}{h^{\alpha}}\frac{\zeta'}{\zeta}(1+\alpha)-\frac{\Lambda(k)}{h^{\alpha}\Phi(1+\alpha,k)}	
	-\frac{k}{\varphi(k)}\Phi(\alpha,k)\zeta(1-\alpha)\frac{\left(\frac{T}{2\pi k}\right)^{-\alpha}}{1-\alpha}\right).
\end{align*}
If $k>1$, then we have 
	\begin{align*}
		\frac{d^{m}}{d\alpha^m}  \mathcal{F}_{\alpha,h,k}(T) \Biggr|_{\alpha=0}=\frac{T}{2\pi}\left((-1)^{m+1}\mathcal{A}_m(h,k)+\mathcal{B}_m(k,T)\right)
	\end{align*}
where $\mathcal{A}_m(h,k)$ and $\mathcal{B}_m(h,k,T)$ are defined by (\ref{definition_mathcal_A}) and (\ref{definition_mathcal_B}), respectively.
\end{lemma}
\begin{proof}
	Let $0<\left|\alpha\right|<1$. For $k>1$, we have 
\begin{align}\label{Curly_F_k>1}
		\mathcal{F}_{\alpha,h,k}(T)=  \frac{T}{2 \pi} 
	\left( -\Lambda(k)
		h^{-\alpha} \Phi(1+\alpha,k)^{-1}
		-\frac{k}{\varphi(k)}\Phi(\alpha,k)\zeta(1-\alpha)\frac{\left(\frac{T}{2\pi k}\right)^{-\alpha}}{1-\alpha} \right).
\end{align}	
Due to the factor $\Lambda(k)$ in the first term inside the parenthesis above, it is sufficient to consider when $k$ is a prime power. Let $k=p^{a}$ where $p$ is a prime number and $a\in\mathbb{N}$.

By the first assertion in Lemma \ref{powerseries}, we have 
 \begin{align*}
	h^{-\alpha} \Phi(1+\alpha,p^a)^{-1}
	& = \sum_{u=0}^{\infty} A_u \alpha^u
\end{align*}
where 
\begin{align*}
		A_u & =   \sum_{\substack{u_1+u_2=u\\u_1,u_2\geqslant 0}} \frac{1}{u_1!} (-1)^{u_1} (\log h)^{u_1} \frac{1}{u_2!}
	 \left(\mathbbm{1}_{ u_2=0}+(-1)^{u_2}  (\log^{u_2} p)\pl_{-u_2}\left(\frac{1}{p}\right)\right)  
		\\& = \frac{(-1)^u}{u!}   \left(\log^u h +  \sum_{\substack{u_1+u_2=u\\u_1,u_2\geqslant 0}} \binom{u}{u_1} (\log^{u_1} h)
		(\log^{u_2} p)\pl_{-u_2}\left(\frac{1}{p}\right)  \right).
\end{align*}

Comparing this with  the definition of $\mathcal{A}_m(h,k)$ in (\ref{definition_mathcal_A}),  we find that
	\begin{align}\label{derivative_of_the_first_term}
			\frac{d^{m}}{d\alpha^m} 	\left(-\Lambda(k)h^{-\alpha} \Phi(1+\alpha,k)^{-1}\right)  \Biggr|_{\alpha=0}=(-1)^{m+1}	\mathcal{A}_m(h,k).
	\end{align}
Now we consider the second term inside the parenthesis in (\ref{Curly_F_k>1}). 
By Lemma \ref{powerseries} and \eqref{Laurent}, we have 
\begin{align*} 
\Phi(\alpha,k)\zeta(1-\alpha)
&=(-1)^{\omega(k)}\left(	\sum_{u={\omega(k)-1}}^{\infty}
\mathcal{G}_{u+1}(k)
\alpha^{u}	\right)\left(\sum_{u=0}^{\infty}\tilde{\gamma}_{u-1}\alpha^u\right).
\end{align*}
From the power series
\begin{equation}
	\left(\frac{T}{2\pi k}\right)^{-\alpha} = \sum_{u=0}^{\infty} \frac{(-1)^u \log^u (\frac{T}{2 \pi k}) }{u!} \alpha^u
	\text{ and }
	\frac{1}{1-\alpha}=\sum_{u=0}^{\infty}\alpha^u,
\end{equation}
	we have 
\begin{align*}
\Phi(\alpha,k)\zeta(1-\alpha)\frac{\left(\frac{T}{2\pi k}\right)^{-\alpha}}{1-\alpha} 
=(-1)^{\omega(k)}\sum_{u=\omega(k)-1}^{\infty}B_u\alpha^{u}
\end{align*}	
	where, for $u\geqslant \omega(k)-1$, we define
	\begin{align*}
		B_u&:= \sum_{\substack{u_1+u_2+u_3+u_4=u\\u_1\geqslant \omega(k)-1\\u_2,u_3,u_4\geqslant 0}} \mathcal{G}_{u_1+1}(k)\tilde{\gamma}_{u_2-1}\frac{(-1)^{u_3}}{u_3!}\log^{u_3}\left(\frac{T}{2\pi k}\right)
		\\&=\sum_{\substack{u_1+u_2\leqslant u\\u_1\geqslant \omega(k)-1\\u_2\geqslant 0}} \mathcal{G}_{u_1+1}(k)\frac{(-1)^{u_2}}{u_2!}\log^{u_2}\left(\frac{T}{2\pi k}\right)\sum_{j=0}^{u-u_1-u_2}\tilde{\gamma}_{j-1}.
	\end{align*}
Recalling the definition of $\mathcal{B}_m(k,T)$ in (\ref{definition_mathcal_B}), it follows that

	\begin{align}\label{derivative_second_term}
	\frac{d^{m}}{d\alpha^m} \left(-\frac{k}{\varphi(k)}\Phi(\alpha,k)\zeta(1-\alpha)\frac{\left(\frac{T}{2\pi k}\right)^{-\alpha}}{1-\alpha}\right)   \Biggr|_{\alpha=0}=	\mathcal{B}_m(k,T)
	\end{align}
	since the derivative above is zero if $\omega(k)>m+1.$
	By (\ref{derivative_of_the_first_term}) and (\ref{derivative_second_term}), we obtain the desired result.	
\end{proof}

\begin{lemma} \label{errorbounds}
Let $\alpha\in\mathbb{C}$ such that $\left|\alpha\right|\leqslant \frac{1}{15\log T}$ and $s_{-\alpha}(n)=n^{-\alpha}$. Let $\kappa=1+\frac{1}{\log T}$.  For sequences $x(n)$ and $y(n)$ supported on $n\leqslant N\ll T^{\vartheta}$ for some $\vartheta<1/2$, we have 
\begin{align}
 \sum_{n= 1}^{\infty}\frac{\left|\left(s_{-\alpha}\ast x\right) (n)\right|^2}{n^{2\kappa-1}} & \ll N\log T\left\|\frac{x(n)}{n} \right\|_1^2,  
 \label{bound1} \\
 \sum_{n\leqslant N}\frac{\left|y(n)\right|^2}{n^{1-2\kappa}} & \ll N\left\|y \right\|_1^2,
 \label{bound2} \\
 \sum_{n= 1}^{\infty}\frac{\left|\left(\Lambda \ast s_{-\alpha}\ast x\right) (n)\right|^2}{n^{2\kappa-1}}  
 & \ll  N\log^3 T\left\| \frac{x(n)}{n} \right\|_1^2 \label{bound3}. 
\end{align}

\end{lemma}
\begin{proof}
Since $\left|\alpha\right|\leqslant \frac{1}{15\log T}$, we  have 
\begin{align*}
	\sum_{n= 1}^{\infty}&\frac{\left|\left(s_{-\alpha}\ast x\right) (n)\right|^2}{n^{2\kappa-1}} 
	=\sum_{n=1}^{\infty}\frac{1}{n^{1+\frac{2}{\log T}}}\sum_{d_1,d_2\mid n}\left(n/d_1\right)^{-\alpha}\left(n/d_2\right)^{-\overline{\alpha}}x\left(d_1\right)\overline{x(d_2)}
	\\&\leqslant  \sum_{d_1,d_2\leqslant N}\left|x(d_1)x(d_2)\right|\sum_{\substack{n=1\\ d_1,d_2\mid n}}^{\infty}\frac{\left(\frac{n^2}{d_1d_2}\right)^{\frac{1}{2\log T}}}{n^{1+\frac{2}{\log T}}}
	\leqslant \sum_{d_1,d_2\leqslant N}\left|x(d_1)x(d_2)\right|\sum_{\substack{n=1\\ d_1,d_2\mid n}}^{\infty}\frac{1}{n^{1+\frac{1}{\log T}}}
	\\&= \sum_{d_1,d_2\leqslant N}\left|x(d_1)x(d_2)\right|\left(\frac{(d_1,d_2)}{d_1d_2}\right)^{1+\frac{1}{\log T}}\sum_{k=1}^{\infty}\frac{1}{k^{1+\frac{1}{\log T}}}
	\ll  N\log T\left\|\frac{x(n)}{n} \right\|_1^2. 
\end{align*}
Since $N\leqslant T^{\vartheta}$, we have 
\begin{align*}
	\sum_{n\leqslant N}\frac{\left|y(n)\right|^2}{n^{1-2\kappa}}=\sum_{n\leqslant N}\left|y(n)\right|^2 n^{1+\frac{2}{\log T}}\ll \sum_{n\leqslant N}\left|y(n)\right|^2 n&\leqslant \left(\sum_{n\leqslant N}\left|y(n)\right|n^{1/2}\right)^2
	\leqslant N\left\|y \right\|_1^2.
\end{align*}
Finally, we have
\begin{align*}
	\nonumber\sum_{n= 1}^{\infty}\frac{\left|\left(\Lambda \ast s_{-\alpha}\ast x\right) (n)\right|^2}{n^{2\kappa-1}} 
	&\nonumber \leqslant \sum_{n=1}^{\infty}\frac{1}{n^{1+\frac{2}{\log T}}}\sum_{\substack{d_1,d_2,d_3,d_4\\d_1\mid n, \, d_2\mid n/d_1\\d_3\mid n,\, d_4\mid n/d_3}}\Lambda(d_1) \Lambda(d_3) \left|x(d_2)x(d_4)\right|\left(\frac{n^2}{d_1d_2d_3d_4}\right)^{1/2\log T}
	\\&\nonumber\leqslant \sum_{d_2,d_4\leqslant N}\left|x(d_2)x(d_4)\right|\sum_{d_1=1}^{\infty}\Lambda(d_1)\sum_{d_3=1}^{\infty}\Lambda(d_3)\sum_{\substack{n=1\\d_1d_2\mid n\\ d_3d_4\mid n}}^{\infty}\frac{1}{n^{1+\frac{1}{\log T}}}.
	\end{align*}
For given $d_1,d_2,d_3,d_4$, with $d_2,d_4\leq N$, we have 
\begin{align*}
\sum_{\substack{n=1\\d_1d_2\mid n\\ d_3d_4\mid n}}^{\infty}\frac{1}{n^{1+\frac{1}{\log T}}}= \frac{(d_1d_2,d_3d_4)}{(d_1d_2d_3d_4)^{1+\frac{1}{\log T}}}\sum_{k=1}^{\infty}\frac{1}{k^{1+\frac{1}{\log T}}}\leqslant \log T \frac{(d_1d_2,d_3d_4)}{(d_1d_2d_3d_4)^{1+\frac{1}{\log T}}}
\end{align*}
and thus
\begin{align}\label{upperboundwithlambda}
\sum_{n= 1}^{\infty}\frac{\left|\left(\Lambda \ast s_{-\alpha}\ast x\right) (n)\right|^2}{n^{2\kappa-1}} \ll \log T  \sum_{d_2,d_4\leqslant N}\left|x(d_2)x(d_4)\right|\sum_{d_1=1}^{\infty}\Lambda(d_1)\sum_{d_3=1}^{\infty}\Lambda(d_3)\frac{(d_1d_2,d_3d_4)}{(d_1d_2d_3d_4)^{1+\frac{1}{\log T}}}.
\end{align}
Since $\Lambda(d_1)$ and $\Lambda(d_3)$ are supported on prime powers and $d_2,d_4\leqslant N$, we have 
\begin{align}\label{upperboundforthesummand}
	\Lambda(d_1)\Lambda(d_3)\frac{(d_1d_2,d_3d_4)}{(d_1d_2d_3d_4)^{1+\frac{1}{\log T}}}&\leqslant \Lambda(d_1)\Lambda(d_3)\frac{N}{(d_1d_2d_3d_4)^{1+\frac{1}{\log T}}}
	+\mathbbm{1}_{\substack{d_1=p^{a}\\d_3=p^b}}\log^{2}p \frac{N\min\lbrace p^{a}, p^{b}\rbrace }{(d_2d_4p^ap^b)^{1+\frac{1}{\log T}}}.
\end{align}
The contribution of the first term above to the right hand side of (\ref{upperboundwithlambda}) is bounded above by
\begin{align*}
	&\log T	\sum_{d_2,d_4\leqslant N}\left|x(d_2)x(d_4)\right|\sum_{d_1=1}^{\infty}\sum_{d_3=1}^{\infty} \Lambda(d_1)\Lambda(d_3)\frac{N}{(d_1d_2d_3d_4)^{1+\frac{1}{\log T}}}
	\ll N\log^3 T \left\| \frac{x(n)}{n} \right\|_1^2.
\end{align*}
The contribution of the second term in (\ref{upperboundforthesummand}) to the right hand side of (\ref{upperboundwithlambda}) is bounded above by
\begin{align*}
&N\log T\left\| \frac{x(n)}{n} \right\|_1^2 \sum_{\substack{p, a\geqslant1}}\frac{\log^2p}{p^{a\left(1+\frac{1}{\log T}\right)}}\left( \sum_{1\leqslant b \leqslant a} \frac{p^b}{p^{b\left(1+\frac{1}{\log T}\right)}}+p^{a}\sum_{b\geqslant a+1}\frac{1}{p^{b\left(1+\frac{1}{\log T}\right)}}\right)
	\\&\ll N\log T\left\| \frac{x(n)}{n} \right\|_1^2 \sum_{\substack{p, a\geqslant1}}\frac{\log^2p}{p^{a\left(1+\frac{1}{\log T}\right)}} \left( a+\frac{1}{p} \right)
	\\&\ll  N\log T\left\| \frac{x(n)}{n} \right\|_1^2\sum_{n=1}^{\infty}\frac{\Lambda(n)\log n}{n^{1+\frac{1}{\log T}}}
	\ll N\log^3 T\left\| \frac{x(n)}{n} \right\|_1^2. 
\end{align*}
\end{proof}

\begin{lemma}\label{lemma_for__moment_corollary}
	Let $k,u\geqslant 1$ be natural numbers. We have
	\begin{align*}
		\sum_{j=1}^{u}\sum_{\substack{p_1,p_2\dots, p_j\\p_1^{a_1}\dots p_j^{a_j}\leqslant N}}\frac{\tau_{k}\left(p_1^{a_1}\dots p_j^{a_j}\right)}{p_1^{a_1}\dots p_j^{a_j}}\left(\log p_1+\log p_2+\dots +\log p_j\right)^{u}\ll \log^{u} N
	\end{align*}
	where the sum over $p_1,\dots, p_j$ runs over distinct prime numbers and  the implied constant depends on $k$ and $u$. 
\end{lemma}
\begin{proof}
	We use induction $u\geqslant 1$. For the base case when $u=1$, we have 
	\begin{align*}
		\sum_{p^a\leqslant N}\frac{\tau_{k}(p^a)}{p^a}\log p\ll\log N
	\end{align*} 
	since the main contribution comes from the case when $a=1$ in which case $\tau_k(p)=k$. Assume that the desired bound holds for some $u\geqslant 1$. Let 
	\begin{align*}
		\sum_{j=1}^{u+1}\sum_{\substack{p_1,p_2\dots, p_j\\p_1^{a_1}\dots p_j^{a_j}\leqslant N}}\frac{\tau_{k}\left(p_1^{a_1}\dots p_j^{a_j}\right)}{p_1^{a_1}\dots p_j^{a_j}}\left(\log p_1+\log p_2+\dots +\log p_j\right)^{u+1}=S_1+S_2
	\end{align*}
	where $S_1$ denotes the sum of the terms $1 \le j \le u$ and $S_2$ is the term corresponding to $j=u+1$. 

	We have 
	\begin{align*}
		S_1&\ll u\log N 	\sum_{j=1}^{u}\sum_{\substack{p_1,p_2\dots, p_j\\p_1^{a_1}\dots p_j^{a_j}\leqslant N}}\frac{\tau_{k}\left(p_1^{a_1}\dots p_j^{a_j}\right)}{p_1^{a_1}\dots p_j^{a_j}}\left(\log p_1+\log p_2+\dots +\log p_j\right)^{u}
		\ll \log^{u+1}N
	\end{align*}
	by the induction hypothesis. For $S_2$, we have 
	\begin{align*}
		S_2& \leqslant\sum_{\substack{j_1+\dots +j_{u+1}=u+1\\j_1,\dots,j_{u+1}\geqslant 0}}\binom{u+1}{j_1, j_2, \dots, j_{u+1}}\prod_{i=1}^{u+1} \left(\sum_{p^a\leqslant N}\frac{\tau_{k}(p_i^a)}{p_i^a}\log^{j_{i}}p_i\right) 
		\\& \ll \sum_{\substack{j_1+\dots +j_{u+1}=u+1\\j_1,\dots,j_{u+1}\geqslant 0}}\binom{u+1}{j_1, j_2, \dots, j_{u+1}}\prod_{i=1}^{u+1}C_{j_i,k}\log^{j_i}N
		\ll \log^{u+1}N
	\end{align*}
	where the implied constant depends on $k$ and $u$. Combining the upper bounds for $S_1$ and $S_2$ finishes the proof.
\end{proof}

\section{Proof of Theorem \ref{mainthm}}

Let $T$ be a large real number and without loss of generality assume that $T$ satisfies the condition in \eqref{eq:Tcondition}. Let $\kappa:=1+\frac{1}{\log T}$ and $\mathscr{C}$ be the positively oriented rectangle with vertices at $1-\kappa+i,\kappa+i,\kappa+iT$ and $1-\kappa+iT$. Then by the functional equation $\zeta(s+\alpha) = \chi(s+\alpha) \zeta(1-s-\alpha)$ and the  residue theorem, we have 
\begin{align*}
	\mathcal{S} :=S(\alpha,T,X,Y)&=
	\sum_{0 < \gamma \leq T} \zeta(\rho+\alpha)X(\rho) Y(1\!-\! \rho)\\& = -  \frac{1}{2 \pi i} \int_{\mathscr{C}} \frac{\zeta'}{\zeta}(1\!-\!s) \chi(s+\alpha) \zeta(1\!-\! s-\alpha) X(s) Y(1\!-\! s) \, ds.
\end{align*}
By the convexity bounds for the Riemann zeta function and the assumption that $\left|\alpha\right|\ll \frac{1}{\log T}$, we have $\chi(s+\alpha) \zeta(1-s-\alpha)=\zeta(s+\alpha)\ll T^{\frac{1}{2}+\varepsilon}$ for $s$ lying on the horizontal parts of the contour $\mathscr{C}$.  Since $T$ satisfies (\ref{eq:Tcondition}), the contribution of the horizontal parts of the integral above is 
\begin{align}\label{horizontal contribution}
	\nonumber\ll T^{\frac{1}{2}+\varepsilon} \sum_{n,m\leqslant N}\frac{\left|x(n)y(m)\right|}{m}\int_{-\frac{1}{\log T}}^{1+\frac{1}{\log T}}\left(\frac{n}{m}\right)^{-\sigma}\, d\sigma
	&\ll T^{\frac{1}{2}+\varepsilon} \sum_{n,m\leqslant N}\frac{\left|x(n)y(m)\right|}{m}\left( 1+\frac{m}{n} \right)
	\\	&\ll T^{\frac{1}{2}+\varepsilon} \left(\left\| x \right\|_{1}\left\| \frac{y(n)}{n} \right\|_{1}+\left\| y \right\|_{1}\left\| \frac{x(n)}{n} \right\|_{1}\right).
\end{align}
Let $\mathcal{S}_R$ and $\mathcal{S}_L$	denote the integrals over the right-hand side and the left-hand side of the contour $\mathcal{C}$, respectively. We have 
\begin{align}\label{SequalsSRSLerror}
	\mathcal{S}=\mathcal{S}_R+\mathcal{S}_L+O\left( \textswab{e}_1\right)
\end{align}
where $\textswab{e}_1$ is a function of $T$ with 
\begin{equation}
  \label{SRerror}
\textswab{e}_1\ll T^{\frac{1}{2}+\varepsilon} \left(\left\| x \right\|_{1}\left\| \frac{y(n)}{n} \right\|_{1}+\left\| y \right\|_{1}\left\| \frac{x(n)}{n} \right\|_{1}\right).
\end{equation}
We consider $	\mathcal{S}_R$ by using 
\begin{align}\label{logarithmicderivativeofthefunctionaleq}
	\frac{\zeta'}{\zeta}(1\!-\! s) = \frac{\chi'}{\chi}(s) - \frac{\zeta'}{\zeta}(s)
	=- \log \Big( \frac{|t|}{2 \pi} \Big)- \frac{\zeta'}{\zeta}(s) + O\left(|t|^{-1}\right).
\end{align}
Since \begin{align*}\chi(s+\alpha) \zeta(1-s-\alpha)=\zeta(s+\alpha)\ll\zeta\left(\kappa+\Re(\alpha)\right)\ll \zeta\left(1+\frac{1}{2\log T}\right)\ll \log T 
\end{align*} 
for $s=\kappa+it$, $1\leqslant t \leqslant T$	 and $\left|\alpha\right|\leqslant \frac{1}{15\log T}$, 
the contribution of the error term in (\ref{logarithmicderivativeofthefunctionaleq}) to $\mathcal{S}_R$ is 
\begin{align}\label{errorinSR}
	\ll \log T\sum_{n,m\leqslant N}\frac{\left|x(n)y(m)\right|}{n} \int_{1}^{T}\frac{dt}{t} \ll \log^2 T \left\| y \right\|_{1}\left\| \frac{x(n)}{n} \right\|_{1}
\end{align}
which is absorbed by the error term $\textswab{e}_1$ in (\ref{SequalsSRSLerror}). 
The contribution of the first term on the right-hand side of (\ref{logarithmicderivativeofthefunctionaleq}) to $\mathcal{S}_R$ is 
\begin{align}\label{firstcontributiontoSr}
	\frac{1}{2 \pi } \int_{1}^{T} \log\left(\frac{t}{2\pi}\right) \zeta(\kappa+it+\alpha) X(\kappa+it) Y(1-\kappa-it) \, dt.
\end{align}
Recalling that $x(n)=y(n)=0$ for $n>N$, we apply Lemma \ref{smoothmv} where  we take $	g(t)=\log\left(t/2\pi\right), a_n=\frac{\left(s_{-\alpha}\ast x\right) (n)}{n^{\kappa}}$ and $ b_n=\frac{y(n)}{n^{1-\kappa}}$. Since $	\int_{1}^{T}\log \left(\frac{t}{2\pi}\right)\, dt=T\log\left(\frac{T}{2\pi e}\right)+\log(2\pi e)$, the integral in (\ref{firstcontributiontoSr} is 
\begin{align}\label{firstterminSR}
	\nonumber&\frac{T}{2\pi}\log\left(\frac{T}{2\pi e}\right)\sum_{n\leqslant N}\frac{\left(s_{-\alpha}\ast x\right) (n)y(n)}{n}+O\left(\left|\sum_{n\leqslant N}\frac{\left(s_{-\alpha}\ast x\right) (n)y(n)}{n}\right| \right)
	\\&\quad \quad +O\left( \log T \left( \sum_{n= 1}^{\infty}\frac{\left|\left(s_{-\alpha}\ast x\right) (n)\right|^2}{n^{2\kappa-1}} \right)^{1/2}  \left( \sum_{n\leqslant N}\frac{\left|y(n)\right|^2}{n^{1-2\kappa}} \right)^{1/2}           \right).
\end{align}
Since 
\begin{align*}
	\left|\sum_{n\leqslant N}\frac{\left(s_{-\alpha}\ast x\right) (n)y(n)}{n}\right|&\leqslant 	\left( \sum_{n\leqslant N}\frac{ \left|\left(s_{-\alpha}\ast x\right)(n)\right|^2 }{n^{2\kappa-1}}     \right)^{1/2}\left( \sum_{n\leqslant N}\frac{\left|y(n)\right|^2}{n^{3-2\kappa}}    \right)^{1/2}
	\\&\leqslant 	\left( \sum_{n\leqslant N}\frac{ \left|\left(s_{-\alpha}\ast x\right)(n)\right|^2 }{n^{2\kappa-1}}     \right)^{1/2}\left( \sum_{n\leqslant N}\frac{\left|y(n)\right|^2}{n^{1-2\kappa}}    \right)^{1/2},
\end{align*}
it suffices to bound the second error term in (\ref{firstterminSR}). 

By Lemma \ref{errorbounds}, \eqref{bound1} and \eqref{bound2}, the second error term in (\ref{firstterminSR}) is $$\ll N\left(\log T\right)^{3/2}\left\|\frac{x(n)}{n} \right\|_1\left\|y \right\|_1 \ll \textswab{e}_1$$
where $\textswab{e}_1$ is the error term in (\ref{SequalsSRSLerror}).

The contribution of the second term on the right-hand side of (\ref{logarithmicderivativeofthefunctionaleq}) to $\mathcal{S}_R$ is 
\begin{align*}
	\frac{1}{2 \pi } \int_{1}^{T} \frac{\zeta'}{\zeta}(\kappa+it)  \zeta(\kappa+it+\alpha) X(\kappa+it) Y(1-\kappa-it) \, dt.
\end{align*}
Applying Lemma \ref{smoothmv} again where we take $g(t)=1, a_n=\frac{\left(-\Lambda\ast s_{-\alpha}\ast x\right) (n)}{n^{\kappa}}, b_n=\frac{y(n)}{n^{1-\kappa}}$,  the integral above equals 
\begin{align}\label{secondterminSR}
	\nonumber -\frac{T}{2\pi}\sum_{n\leqslant N}&\frac{\left(\Lambda\ast s_{-\alpha}\ast x\right) (n)y(n)}{n}
	+O\left(\left|\sum_{n\leqslant N}\frac{\left(\Lambda\ast s_{-\alpha}\ast x\right) (n)y(n)}{n}\right|\right)
	\\& 	 +O\left(    \left(\sum_{n=1}^{\infty}\frac{\left|\left(\Lambda\ast s_{-\alpha}\ast x\right) (n)\right|^2}{n^{2\kappa-1}}  \right)^{1/2}    \left(\sum_{n\leqslant N}\frac{\left|y(n)\right|^2}{n^{1-2\kappa}}  \right)^{1/2}       \right).
\end{align}
By Lemma \ref{errorbounds}, we have
\begin{align*}
	\left|\sum_{n\leqslant N}\frac{\left(\Lambda\ast s_{-\alpha}\ast x\right) (n)y(n)}{n}\right|&\leqslant	\left( \sum_{n\leqslant N}\frac{ \left|\left(\Lambda\ast s_{-\alpha}\ast x\right)(n)\right|^2 }{n^{2\kappa-1}}     \right)^{1/2}\left( \sum_{n\leqslant N}\frac{\left|y(n)\right|^2}{n^{3-2\kappa}}    \right)^{1/2}
	\\&\leqslant 	\left( \sum_{n\leqslant N}\frac{ \left|\left(\Lambda\ast s_{-\alpha}\ast x\right)(n)\right|^2 }{n^{2\kappa-1}}     \right)^{1/2}\left( \sum_{n\leqslant N}\frac{\left|y(n)\right|^2}{n^{1-2\kappa}}    \right)^{1/2} \\
	& \ll  
	N(\log T)^{3/2} \left\|\frac{x(n)}{n} \right\|_1\left\|y \right\|_1 \ll \textswab{e}_1
\end{align*}
and thus the error terms in (\ref{secondterminSR}) are absorbed by $\textswab{e}_1$
which gives the following result for $\mathcal{S}_R$.
\begin{theorem} \label{SRtheorem}
	If $x(n)$ and $y(n)$ satisfy \eqref{xyinftyunc}, \eqref{submultiplicative}, 
$\alpha$ satisfies \eqref{alpha size condition} and $T$ satisfies \eqref{eq:Tcondition}, then 
\begin{align}\label{final_S_R}
	\mathcal{S}_R=\frac{T}{2\pi}\log\left(\frac{T}{2\pi e}\right)\sum_{n\leqslant N}\frac{\left(s_{-\alpha}\ast x\right) (n)y(n)}{n}-\frac{T}{2\pi}\sum_{n\leqslant N}\frac{\left(\Lambda\ast s_{-\alpha}\ast x\right) (n)y(n)}{n}+O\left(\textswab{e}_1\right)
\end{align} 
where $\textswab{e}_1$ is defined in \eqref{SRerror}.   
\end{theorem}

 Next, we estimate $\mathcal{S}_L$. We have 
	\begin{align*}
	\mathcal{S}_L&=	- \frac{1}{2 \pi i} \int_{1-\kappa+iT}^{1-\kappa+i}\frac{\zeta'}{\zeta}(1\!-\!s) \chi(s+\alpha) \zeta(1\!-\! s-\alpha) X(s) Y(1\!-\! s) \, ds
	\\&=\overline{\frac{1}{2\pi}\int_{1}^{T}\frac{\zeta'}{\zeta}(\kappa+it)\chi(1-\kappa-it+\overline{\alpha})\zeta(\kappa+it-\overline{\alpha})\overline{X}(1-\kappa-it)\overline{Y}(\kappa+it)\, dt}
	\end{align*}
where we use the notation $\displaystyle{\overline{X}(s)=\sum_{n\leqslant N}\overline{x(n)}/n^s}$ and $\displaystyle{\overline{Y}(s)=\sum_{n\leqslant N}\overline{y(n)}/n^s}$.	
Define
	\begin{align*}
			\mathcal{S}_L(\gamma):=\frac{1}{2\pi}\int_{1}^{T}\frac{\zeta(\kappa+\gamma+it)}{\zeta(\kappa+it)}\chi(1-\kappa-it+\overline{\alpha})\zeta(\kappa+it-\overline{\alpha})\overline{X}(1-\kappa-it)\overline{Y}(\kappa+it)\, dt
	\end{align*}
	for $\gamma\in \mathbb{C}$ with $\left|\gamma\right|\leqslant \frac{1}{15\log T}$.
	Then we have
\begin{equation}
   \label{SLidentity}
		\mathcal{S}_L=\overline{\frac{d}{d\gamma}	\mathcal{S}_L(\gamma)\Big|_{\gamma=0}}.
\end{equation}
Now we consider  $	\mathcal{S}_L(\gamma)$. By the change of variable $s=\kappa+it=w+\overline{\alpha}$, we have 
	\begin{align*}
			\mathcal{S}_L(\gamma)
			=\frac{1}{2\pi i}\int_{\kappa+i-\overline{\alpha}}^{\kappa+iT-\overline{\alpha}}\frac{\zeta(w+\overline{\alpha}+\gamma)}{\zeta(w+\overline{\alpha})}\chi(1-w)\zeta(w)\overline{X}(1-w-\overline{\alpha})\overline{Y}(w+\overline{\alpha})\, dw.
	\end{align*}
Since $\left|\alpha\right|$ and $\left|\gamma\right|$ are small enough so that $\Re w>1$,  we can make use of the underlying Dirichlet series of the integrand above. For $w$ being consider in the integral above, 	define
\begin{align*}
	A(w):&=\frac{\zeta(w+\overline{\alpha}+\gamma)}{\zeta(w+\overline{\alpha})}\zeta(w)\overline{Y}(w+\overline{\alpha})
	=\sum_{m=1}^{\infty}\frac{a(m)}{m^{w}}
\end{align*}
and 
\begin{align} \label{Bdefn}
		B(1-w):=\overline{X}(1-w-\overline{\alpha})	=\sum_{k\leqslant N}\frac{b(k)}{k^{1-w}}
\end{align}
	where 
	\begin{align}\label{definition of a as convolution}
		a(m)=a(m,\gamma,\alpha):=\sum_{\substack{m_1m_2m_3m_4=m\\m_4\leqslant N}}\mu(m_1)m_1^{-\overline{\alpha}}m_2^{-\overline{\alpha}-\gamma}\overline{y(m_4)}m_4^{-\overline{\alpha}},
	\end{align}
and
\begin{align}\label{bdefn}
		b(k)=b(k,\alpha):=\overline{x(k)}k^{\overline{\alpha}}.
\end{align}
Then we have 
\begin{align*}
	\mathcal{S}_L(\gamma)=\frac{1}{2\pi i}\int_{\kappa+i-\overline{\alpha}}^{\kappa+iT-\overline{\alpha}}\chi(1-w)B(1-w)A(w)\,dw.
\end{align*}

Since $\left|\alpha\right|, \left|\gamma\right|\leqslant 1/\left(15\log T\right)$ and $N\leqslant T^\vartheta$ for some $\vartheta<1/2$, we have 
\begin{align*}
	a(m)\ll m^{\frac{1}{5\log T}}(\tau_3\ast \left|y\right|)(m).
\end{align*}
Thus, by Lemma \ref{CGG_type_lemma}, we have 
\begin{align*}
	\mathcal{S}_L(\gamma)=\sum_{k\leqslant N}\frac{b(k)}{k}\sum_{m\leqslant kT/2\pi}a(m)e\left(-m/k\right)+O\left(  T^{1/2}\log^3 T\left\|x \right\|_1\left\|\frac{y(n)}{n} \right\|_1    \right).
\end{align*}
Note the the error term above is absorbed by the error term $\textswab{e}_1$ in (\ref{SRerror}).

Let $m'=m/(m,k)$ and $k'=k/(m,k)$. We have the 
identity
\begin{equation}
	e \left(-\frac{m}{k} \right) = e \left(-\frac{m'}{k'} \right)
	= \frac{\mu(k')}{\phi(k')}
	+ \sum_{{\begin{substack}{q \mid k
					\\ q > 1}\end{substack}}}  \ \chiqq \tau(\overline{\psi})
	\sum_{{\begin{substack}{d \mid m
					\\ d  \mid k}\end{substack}}} 
	\psi \left( \frac{m}{d}\right)  \delta(q,k,d,\psi)
\end{equation}
where $\psi$ ranges through primitive Dirichlet characters, $\tau(\overline{\psi})$ denotes the Gauss sum, and
\begin{align}	\label{eq:delta}
	\delta(q,k,d,\psi) = 
	\sum_{{\begin{substack}{e \mid d
					\\ e \mid k/q}\end{substack}}} 
	\frac{\mu(d/e)}{\phi(k/e)} \overline{\psi}
	\Big( - \frac{k}{eq}
	\Big) \psi \left(\frac{d}{e} \right) \mu 
	\left(\frac{k}{eq}\right).
\end{align}
Note that a proof of this identity may be found in \cite{CGG}, \cite{Ng}. 
Let 
\begin{align*}
	\mathcal{M(\gamma)}:=\sum_{k\leqslant N}\frac{b(k)}{k}\sum_{m\leqslant kT/2\pi}a(m)\frac{\mu(k/(m,k))}{\phi(k/(m,k))}
\end{align*}
and 
\begin{align*}
	\mathcal{E(\gamma)}:=\sum_{k\leqslant N}\frac{b(k)}{k}\sum_{m\leqslant kT/2\pi}a(m)\sum_{\substack{q\mid k \\ q>1}}   \chiqq \tau(\overline{\psi})
	\sum_{\substack{d\mid m \\ d\mid k}} 
	\psi \left(\frac{m}{d}\right) \delta(q,k,d,\psi).
\end{align*}
Thus  we have 
	\begin{align*}
			\mathcal{S}_L(\gamma)=	\mathcal{M}(\gamma)+\mathcal{E}(\gamma) +O( \textswab{e}_1).
	\end{align*}
where $ \textswab{e}_1$ is defined in (\ref{SRerror}). 
Since the error term $\textswab{e}_1$ is uniform in $\gamma$ with $\left|\gamma\right|\leqslant 1/\left(15\log T\right)$, an application of the Cauchy Integral Formula to the term $\textswab{e}_1$ produces an error term $\ll \textswab{e}_1\log T$ which is absorbed by the error term $\textswab{e}_1$ due to the factor of $T^\varepsilon$ in $\textswab{e}_1$. Thus, by \eqref{SLidentity}, we have 
\begin{equation}
	\label{SLexpansion}
	\mathcal{S}_L =\mathcal{M}+ \mathcal{E}+O\left(\textswab{e}_1\right)
\end{equation}
where
\begin{equation}\label{MgammatoM}
	\mathcal{M} :=\overline{\frac{d}{d\gamma} \mathcal{M}(\gamma)\Bigr|_{\gamma=0}}
\,\,	\text{ and }\,\, 
	\mathcal{E} :=\overline{\frac{d}{d\gamma} \mathcal{E}(\gamma)\Bigr|_{\gamma=0}}.
\end{equation}

Now, our aim is to estimate the term $\mathcal{M}$ and to find an upper bound for the size of the error term $\mathcal{E}$. In the following sections,  we prove the following  theorems. 
\begin{theorem}\label{Mthm}
Let $Z_{\alpha,\gamma,h,k}$ be defined by (\ref{definition of Z}). We have 
 
\begin{align*}
  \mathcal{M} 
  &= \sum_{g\leqslant N}\sum_{\substack{h,k\leqslant N/g\\(h,k)=1}}\frac{y(gh)x(gk)}{ghk}
  \overline{ \frac{d}{d\gamma }\frac{1}{2\pi}
  	\int_{0}^{T}\left(Z_{\overline{\alpha},\gamma,h,k}+\left(\frac{t}{2\pi}\right)^{-\overline{\alpha}-\gamma} Z_{-\gamma,-\overline{\alpha},h,k} \right) \, dt \Big|_{\gamma=0}  }
  \\&\quad \quad  +O\left(\textswab{e}_2\right)
\end{align*}
where 
\begin{align*}
	\textswab{e}_2:=\begin{cases}
		T\exp\left(-C\sqrt{\log T}\right) &  \text{assuming (\ref{xyinftyunc})},
		\\  T^{\Theta+\varepsilon} \left\|\frac{x(n)y(n)}{n} \right\|_1 \left\|\frac{y(n)}{n^{\Theta}} \right\|_1 \left\|\frac{x(n)}{n} \right\|_1\left\|\frac{x(n)}{n^{2-\Theta}} \right\|_1  & \text{assuming GRH($\Theta$) and (\ref{submultiplicative})}
	\end{cases}
\end{align*}
for some positive constant $C$.  
\end{theorem}
The bound for the error term will be based on   Heath-Brown's combinatorial decomposition of $\mu(n)$ 
and an
application of the large sieve inequality.  The argument combines ideas from  \cite{BHB} and \cite{HLZ}.  

\begin{theorem}
	\label{ErrorUnconditional}
 Let  $A>0$ be arbitrarily large but fixed. Under the assumptions in  (\ref{xyinftyunc}), we have  
	\[
	\mathcal{E} \ll T (\log T)^{-A}. 
	\]
\end{theorem}

\begin{theorem}
\label{EGRH}
Assume GRH($\Theta$).  Then 
\[
  \mathcal{E} \ll T^{\Theta+\varepsilon}  \left\|\frac{y(n)}{n^\Theta}\right\|_1\left\|n^{1/2}x(n)(1\ast\left| y\right|)(n)\right\|_1.
\]
\end{theorem}

Now, by using the theorems above, we conclude the proof of Theorem  \ref{mainthm}. The main term in Theorem \ref{mainthm} is obtained by Theorems \ref{SRtheorem} and \ref{Mthm} and Lemma  \ref{lemma_simplification_of_F}.  
The error term $\tilde{\mathcal{E}}$ in Theorem \ref{mainthm} is obtained by the sum of the error term $\textswab{e}_1$ in (\ref{SequalsSRSLerror}) and in Theorem \ref{SRtheorem}, the error term $\textswab{e}_2$ in Theorem \ref{Mthm} and the error term $\mathcal{E}$ in Theorems \ref{ErrorUnconditional} and \ref{EGRH}. Assuming the bounds in (\ref{xyinftyunc}), we have $\textswab{e}_1\ll T(\log T)^{-A}$ since $N\ll T^{\vartheta}$ for some $\vartheta<1/2$. Thus the desired result follows in this case. On the other hand, assuming (\ref{submultiplicative}) and GRH($\Theta)$, we have 
\begin{align*}
    \tilde{\mathcal{E}}&\ll  T^{\frac{1}{2}+\varepsilon} \left(\left\| x \right\|_{1}\left\| \frac{y(n)}{n} \right\|_{1}+\left\| y \right\|_{1}\left\| \frac{x(n)}{n} \right\|_{1}\right) + T^{\Theta+\varepsilon} \left\|\frac{x(n)y(n)}{n} \right\|_1 \left\|\frac{y(n)}{n^{\Theta}} \right\|_1 \left\|\frac{x(n)}{n} \right\|_1\left\|\frac{x(n)}{n^{2-\Theta}} \right\|_1 
    \\&\quad \quad +T^{\Theta+\varepsilon}  \left\|\frac{y(n)}{n^\Theta}\right\|_1\left\|n^{1/2}x(n)(1\ast\left| y\right|)(n)\right\|_1.
\end{align*}
\qed

\section{Proof of Theorem \ref{Mthm}: Estimate for $\mathcal{M}$} \label{sectionM}
	By the definition of $a(m)$ in \eqref{definition of a as convolution} and using the change of variable $m=nh$ where $h$ stands for the variable $m_4$ (\ref{definition of a as convolution}), we have
	\begin{align*}
		\mathcal{M}(\gamma)=\sum_{k,h\leqslant N}\frac{b(k)\overline{y(h)}h^{-\overline{\alpha}}}{k}\sum_{n\leqslant kT/2\pi h}c(n)\frac{\mu(k/(nh,k))}{\phi(k/(nh,k))}	
	\end{align*}
where 
\begin{align}\label{definition of cn}
	c(n):=\sum_{n_1n_2n_3=n}\mu(n_1)n_1^{-\overline{\alpha}}n_2^{-\overline{\alpha}-\gamma}.
\end{align}
Grouping the terms with the same $g=(h,k)$, we have  	
	\begin{align*}
		\mathcal{M}(\gamma)=	\sum_{g\leqslant N}\sum_{\substack{h,k\leqslant N/g\\(h,k)=1}}\frac{b(gk)\overline{y(gh)}(gh)^{-\overline{\alpha}}}{gk}\sum_{n\leqslant\frac{kT}{2\pi h}}c(n)\frac{\mu(k/(n,k))}{\phi(k/(n,k))}.
	\end{align*}
	Grouping the terms with the same $d=(n,k)$, we have 
	\begin{align}\label{M_gamma_final}
		\mathcal{M}(\gamma)=	\sum_{g\leqslant N}\sum_{h\leqslant N/g}\frac{\overline{y(gh)}(gh)^{-\overline{\alpha}}}{g}\sum_{d\leqslant N/g}\sum_{\substack{k\leqslant N/gd\\(h,dk)=1}}\frac{b(gdk)}{dk}\frac{\mu(k)}{\phi(k)}\sum_{\substack{n\leqslant \frac{kT}{2\pi h }\\(n,k)=1}}c(dn).
	\end{align}
	In Lemma \ref{convolutiondecomposition}, we take $f_1(n):=\mu(n)n^{-\overline{\alpha}}, f_2(n):=n^{-\overline{\alpha}-\gamma}$ and $f_3(n)=1$ for all $n\in\mathbb{N}$. Let $g\leqslant N, \,\,  h,d\leqslant N/g, \, \, k\leqslant N/gd$ and  $\sigma\geqslant 1+\frac{1}{\log T}$. We have 
	\begin{align}\label{decomposition}
	\nonumber	\sum_{\substack{n=1\\(n,k)=1}}^{\infty}&\frac{(f_1\ast f_2\ast f_3)(dn)}{n^{s}}\\&\nonumber =\sum_{d_1d_2d_3=d}\sum_{\substack{n_1=1\\(n_1,kd_1d_2)=1}}^{\infty}\frac{\mu(n_1d_3)(n_1d_3)^{-\overline{\alpha}}}{n_1^s}
	\sum_{\substack{n_2=1\\(n_2,kd_1)=1}}^{\infty}\frac{(n_2d_2)^{-\overline{\alpha}-\gamma}}{n_2^s}
	\sum_{\substack{n_3=1\\(n_3,k)=1}}^{\infty}\frac{1}{n_3^s}
	\\\nonumber&=\sum_{d_1d_2d_3=d}\mu(d_3)d_3^{-\overline{\alpha}}d_2^{-\overline{\alpha}-\gamma}\sum_{\substack{n_1=1\\(n_1,kd_1d_2d_3)=1}}^{\infty}\frac{\mu(n_1)n_1^{-\overline{\alpha}}}{n_1^s}
	\sum_{\substack{n_2=1\\(n_2,kd_1)=1}}^{\infty}\frac{n_2^{-\overline{\alpha}-\gamma}}{n_2^s}
	\sum_{\substack{n_3=1\\(n_3,k)=1}}^{\infty}\frac{1}{n_3^s}
	\\&=\frac{\zeta(s+\overline{\alpha}+\gamma)\zeta(s)}{\zeta(s+\overline{\alpha})}G(s,k,d,\alpha,\gamma)
	\end{align}
where
	\begin{align*}
		G_1(s,k,d,\alpha,\gamma)&:=\left(\prod_{p\mid kd}\left(1-\frac{1}{p^{s+\overline{\alpha}}}\right)^{-1}\right)\prod_{p\mid k}\left(1-\frac{1}{p^s}\right)
		\\&\quad\quad\quad \times \sum_{d_1d_2d_3=d}\mu(d_3)d_3^{-\overline{\alpha}}d_2^{-\overline{\alpha}-\gamma}\left(\prod_{p\mid kd_1}\left(1-\frac{1}{p^{s+\overline{\alpha}+\gamma}}\right)\right).
	\end{align*}
Note that  by changing the roles of $f_2$ and $f_3$ when applying Lemma \ref{convolutiondecomposition}, we can also write
\begin{align}\label{another form of G for the second residue}
	\nonumber G_1(s,k,d,\alpha,\gamma)&=\sum_{d_1d_2d_3=d}\mu(d_3)d_3^{-\overline{\alpha}}d_1^{-\overline{\alpha}-\gamma}\left(\prod_{p\mid kd_1d_2d_3}\left(1-\frac{1}{p^{s+\overline{\alpha}}}\right)^{-1}\right)\left(\prod_{p\mid kd_1}\left(1-\frac{1}{p^s}\right)\right)
	\\&\quad \quad\times \prod_{p\mid k}\left(1-\frac{1}{p^{s+\overline{\alpha}+\gamma}}\right). 
\end{align}
The formulation in (\ref{another form of G for the second residue}) will be useful when dealing with one of the residues later.

Let 
\begin{align*}
	F_1(s):=\frac{\zeta(s+\overline{\alpha}+\gamma)\zeta(s)}{\zeta(s+\overline{\alpha})}	G_1(s,k,d,\alpha,\gamma)\frac{x^s}{s}
\end{align*}
where $
	x=\frac{kT}{2\pi h}$.
Let $ 2\leqslant U$, $H\leqslant T^2$ and $\kappa=1+\frac{1}{\log T}$. By Lemma \ref{perron}, we have 
\begin{align}\label{perron_1}
		\sum_{\substack{n\leqslant \frac{kT}{2\pi h }\\(n,k)=1}}c(dn)&=\frac{1}{2\pi i}\int_{\kappa-iU}^{\kappa+iU}	F_1(s)\, ds+O\left(\sum_{x-\frac{x}{H}\leqslant n\leqslant x+\frac{x}{H}}\left|c(dn)\right|\right)+O\left( \frac{x^{\kappa}H}{U}\sum_{n=1}^{\infty}\left|c(dn)\right|n^{-\kappa}   \right).
\end{align}
Since $x\ll T^{1+\vartheta}$ and $\left|c(dn)\right|\ll (dn)^{\frac{1}{5\log T}}\tau_3(d)\tau_3(n)\ll n^{\frac{1}{5\log T}} \tau_3(d)\tau_3(n)$,  the second error term on the right-hand side of (\ref{perron_1}) is 
\begin{align}\label{error_perron_1}
	\ll \frac{xH \tau_3(d) \log^3T
	}{U}
\end{align}
where the implied constant is absolute. For the first error term on the right-hand side of (\ref{perron_1}), we assume that $H\leqslant x^{1-\varepsilon}$ for some $0<\varepsilon<1/2$ and thus we can use Shiu's short divisor sum bound, \cite[Theorem 2]{S}, and thus 
\begin{align}\label{error_perron_2}
\sum_{x-\frac{x}{H}\leqslant n\leqslant x+\frac{x}{H}}\left|c(dn)\right|\ll \tau_3(d)\sum_{x-\frac{x}{H}\leqslant n\leqslant x+\frac{x}{H}}\tau_3(n)\ll \tau_3(d) \frac{x}{H}\log^2 x.	
\end{align}
By the explicit classical zero-free region,
\cite{Kad}, $\zeta(s)\neq 0$ in the region $\sigma>1-\frac{1}{6\log \left|t\right|}$ for $\left|t\right|\geqslant 2$. Since $\left|\alpha\right|+\left|\gamma\right|\leqslant \frac{2}{15\log T}<\frac{1}{6\log T}$, there exist an absolute positive constant $c$ such that for $\sigma_0:=1-\frac{c}{\log U}$, we have $\Re(\sigma_0+\overline{\alpha}+\gamma)<1$ and $\zeta(s+\overline{\alpha})\neq0$ for $\sigma\geqslant \sigma_0$ and $\left|t\right|\leqslant U$. Thus 
\begin{align*}
\frac{1}{2\pi i}\int_{\kappa-iU}^{\kappa+iU}	F_1(s)\, ds= R_{s=1}+R_{s=1-\overline{\alpha}-\gamma}+\frac{1}{2\pi i}\left( \int_{\kappa+iU}^{\sigma_0+iU}+\int_{\sigma_0+iU}^{\sigma_0-iU}+\int_{\sigma_0-iU}^{\kappa+iU}  \right)		F_1(s)\, ds
\end{align*}
where $R_{s=1}$ and $R_{s=1-\overline{\alpha}}$ denotes the residues of $F_1(s)$ at $s=1$ and $s=1-\overline{\alpha}-\gamma$, respectively. We leave the calculations of 
$R_{s=1}$ and $R_{s=1-\overline{\alpha}-\gamma}$ to the end of this section. Let 
\begin{align}\label{J_integral}
	J_1(\sigma_0):=\frac{1}{2\pi i}\left( \int_{\kappa+iU}^{\sigma_0+iU}+\int_{\sigma_0+iU}^{\sigma_0-iU}+\int_{\sigma_0-iU}^{\kappa+iU}  \right)	\frac{d}{d\gamma }  F_1(s) \Bigr|_{\gamma=0}	\, ds.
\end{align}
and note that 
\begin{align*}
	\frac{d}{d\gamma }  F_1(s)  \Bigr|_{\gamma=0}=\left(\frac{\zeta'}{\zeta}(s+\overline{\alpha})G_1(s,k,d,\alpha,0)+ G_1'(s,k,d,\alpha, 0)\right)\zeta(s)\frac{x^s}{s}
\end{align*}
where $G_1'(s,k,d,\alpha, 0)=\frac{d}{d\gamma } G_1(s,k,d,\alpha,\gamma)\big|_{\gamma=0}$. For $\sigma\geqslant 1/2$, we have 
	\begin{align}\label{divisor bound for G}
	G_1(s,k,d,\alpha,\gamma)\ll \tau_3(d)\prod_{p\mid kd}\left(1+\frac{5}{p^{\sigma}}\right)\ll \tau_3(d)\tau_2(kd)
\end{align}
and $G_1'(s,k,d,\alpha, 0)\ll\tau_3(d)\tau(kd)\log T $ where the implied constants are absolute. 
 By \cite[Theorem 6.7]{MV_book},  the choice of $\sigma_0$ and the assumptions that $\left|\alpha\right|\leqslant \frac{1}{15\log T}$ and $U\leqslant T$, we have 
\begin{align*}
	\frac{\zeta'}{\zeta}(s+\overline{\alpha}),\, \zeta(s)\ll  \log T
\end{align*}
where $s$ runs over the path in (\ref{J_integral}). Thus we have 
\begin{align}\label{error_J_sigma_0}
J_1(\sigma_0) \ll 	\tau_3(d)\tau_2(kd)\left(\log^3T\right) \left(\frac{x}{U}+x^{1-\frac{c}{\log U}}\right).
\end{align}
Applying the Cauchy Integral Formula to the error terms in (\ref{error_perron_1}) and (\ref{error_perron_2}) and  by choosing $U=\exp\left(C_1\sqrt{\log x}\right)\gg\exp\left(\frac{C_1}{\sqrt{2}}\sqrt{\log T}\right)$ for some positive constant $C_1$ and using the assumptions (\ref{xyinftyunc}) on the coefficients $x(n)$ and $y(n)$, we see that the total contribution of the error terms to the expression in (\ref{M_gamma_final}) is $\ll T\exp\left(-C_2\sqrt{\log T}\right)$ for some $C_2>0$. Hence, by (\ref{MgammatoM}), we have 
\begin{align*}
		\mathcal{M}&=	\sum_{g\leqslant N}\sum_{h\leqslant N/g}\frac{y(gh)(gh)^{-\alpha}}{g}\sum_{d\leqslant N/g}\sum_{\substack{k\leqslant N/gd\\(h,dk)=1}}\frac{x(gdk)(gdk)^{\alpha}}{dk}\frac{\mu(k)}{\phi(k)}\overline{\frac{d}{d\gamma}\left(  R_{s=1}+R_{s=1-\overline{\alpha}-\gamma}   \right)\Bigr|_{\gamma=0}}
		\\&\quad \quad + O\left(T\exp\left(-C_2\sqrt{\log T}\right)\right)
\end{align*}
which finishes the proof of the unconditional part of Theorem \ref{Mthm}.

Now we consider the conditional part of Theorem \ref{Mthm}. Assume that $\zeta(s)\neq 0$ for $\sigma\geqslant \Theta$ for some $\frac{1}{2}<\Theta<1$. Let $\varepsilon>0$ be small and take $\sigma_1:=\Theta+\varepsilon<1$. Note that for sufficiently large $T$, we have  $ \varepsilon\geqslant \frac{3}{15\log T}$  so that $\Re(\sigma_1+\overline{\alpha}+\gamma)<1$ and $\Re(\sigma_1+\overline{\alpha})\geqslant \Theta+\frac{\varepsilon}{2}$. Let $J_1(\sigma_1)$ be the corresponding integral in (\ref{J_integral}) where $\sigma_0$ is replaced by $\sigma_1$. By using the arguments in \cite[Section 13.2]{MV_book}, we have 
\begin{align*}
	\frac{\zeta'}{\zeta}(s+\overline{\alpha}),\, \zeta(s)\ll T^{\varepsilon}
\end{align*}
where $s$ runs through the path in $J(\sigma_1)$. Thus we have 
\begin{align}\label{error_term_J_sigma_1}
 J_1(\sigma_1) \ll 
	  T^{\varepsilon}\left( \frac{x}{U}+x^{\Theta+\varepsilon} \right)
\end{align}
since $\varepsilon$ is arbitrary and $	G(s,k,d,\alpha,\gamma), G'(s,k,d,\alpha,0)\ll T^\varepsilon$ as $d,k \leqslant T$. Thus the total contribution of the errors in (\ref{error_perron_1}), (\ref{error_perron_2}) and in (\ref{error_term_J_sigma_1}) is 
\begin{align*}
	\ll  T^{\varepsilon}\left( \frac{x}{U}+x^{\Theta+\varepsilon} +\frac{xH}{U}+\frac{x}{H}\right)\ll T^{\varepsilon}\left(x^{\Theta} +\frac{xH}{U}+\frac{x}{H}\right)
\end{align*}
since $x\ll T^{1+\vartheta}$. Thus choosing $U=x^{2-2\Theta}$ and $H=x^{1-\Theta}$, we see that the total error term in this case is 	$\ll x^{\Theta}T^\varepsilon$. 
The contribution of this error term to the expression in (\ref{M_gamma_final}) is bounded by 
\begin{align}\label{conditional_error_M}
		\nonumber &T^{\Theta+\varepsilon} \sum_{g\leqslant N}\sum_{h\leqslant N/g}\frac{\left|y(gh)\right|}{gh^{\Theta}}\sum_{d\leqslant N/g}\sum_{\substack{k\leqslant N/gd\\(h,dk)=1}}\frac{\left|x(gdk)\right|}{dk^{1-\Theta}}\frac{\left|\mu(k)\right|}{\phi(k)}
		\\\nonumber&\ll 	T^{\Theta+\varepsilon}\sum_{g\leqslant N}\frac{\left|y(g)x(g)\right|}{g}\sum_{h\leqslant N/g}\frac{\left|y(h)\right|}{h^\Theta}\sum_{d\leqslant N/g}\frac{\left|x(d)\right|}{d}\sum_{k\leqslant N/gd}\frac{\left|x(k)\right|}{k^{1-\Theta}\phi(k)}
		\\&\ll T^{\Theta+\varepsilon} \left\|\frac{x(n)y(n)}{n} \right\|_1 \left\|\frac{y(n)}{n^{\Theta}} \right\|_1 \left\|\frac{x(n)}{n} \right\|_1\left\|\frac{x(n)}{n^{2-\Theta}} \right\|_1
\end{align}
by the submultiplicativity assumptions on $x(n)$ and $y(n)$ and the fact that $\frac{1}{\phi(k)}\ll \frac{k^\varepsilon}{k}\ll \frac{T^{\varepsilon}}{k}$ by \cite[Theorem 327]{HardyWright}.

Hence, by (\ref{MgammatoM}), we have 
\begin{align*}
	\mathcal{M}&=	\sum_{g\leqslant N}\sum_{h\leqslant N/g}\frac{y(gh)(gh)^{-\alpha}}{g}\sum_{d\leqslant N/g}\sum_{\substack{k\leqslant N/gd\\(h,dk)=1}}\frac{x(gdk)(gdk)^{\alpha}}{dk}\frac{\mu(k)}{\phi(k)}\overline{\frac{d}{d\gamma}\left(  R_{s=1}+R_{s=1-\overline{\alpha}-\gamma}   \right)\Bigr|_{\gamma=0}}
	\\&\quad \quad \quad + O\left(T^{\Theta+\varepsilon} \left\|\frac{x(n)y(n)}{n} \right\|_1 \left\|\frac{y(n)}{n^{\Theta}} \right\|_1 \left\|\frac{x(n)}{n} \right\|_1\left\|\frac{x(n)}{n^{2-\Theta}} \right\|_1\right)
\end{align*}
under the assumptions GRH$(\Theta)$and (\ref{submultiplicative}). 

In order to finish the proof of Theorem \ref{Mthm}, we now calculate the residues. Without loss of generality, we can assume that the poles at $s=1$ and $s=1-\overline{\alpha}-\gamma$ are distinct by taking $\left|\gamma\right|<\left|\alpha\right|\leqslant \frac{1}{15\log T}$.  We have 
	\begin{align*}
	R_{s=1}=	\frac{\zeta(1+\overline{\alpha}+\gamma)}{\zeta(1+\overline{\alpha})}	G_1(1,k,d,\alpha,\gamma)\frac{kT}{2\pi h }
\end{align*}
Thus the contribution of $R_{s=1}$ to $\mathcal{M}$ is 
\begin{align*}
	\sum_{g\leqslant N}\sum_{\substack{h,k'\leqslant N/g\\(h,k')=1}}\frac{y(gh)x(gk')}{ghk'} \frac{T}{2\pi} \frac{1}{h^{\alpha}{k'}^{-\alpha}}\overline{\frac{\zeta(1+\overline{\alpha}+\gamma)}{\zeta(1+\overline{\alpha})}\frac{d}{d\gamma}\sum_{dk=k'}\frac{k\mu(k)}{\phi(k)}G_1(1,k,d,\alpha,\gamma)\Big|_{\gamma=0}}.
\end{align*}
We have 
\begin{align*}
\sum_{dk=k'}&\frac{k\mu(k)}{\phi(k)}G_1(1,k,d,\alpha,\gamma)
\\&=\prod_{p\mid k'}\left(1-\frac{1}{p^{1+\overline{\alpha}}}\right)^{-1}\sum_{dk=k'}\mu(k)\sum_{d_1d_2d_3=d}\mu(d_3)d_3^{-\overline{\alpha}}d_2^{-\overline{\alpha}-\gamma}\left(\prod_{p\mid kd_1}\left(1-\frac{1}{p^{1+\overline{\alpha}+\gamma}}\right)\right)
\\&=\prod_{p\mid k'}\left(1-\frac{1}{p^{1+\overline{\alpha}}}\right)^{-1}\sum_{d_1d_2d_3d_4=k'}\mu(d_4)\mu(d_3)d_3^{-\overline{\alpha}}d_2^{-\overline{\alpha}-\gamma}\prod_{p\mid d_1d_4}\left(1-\frac{1}{p^{1+\overline{\alpha}+\gamma}}\right)
\\&=\prod_{p\mid k'}\left(1-\frac{1}{p^{1+\overline{\alpha}}}\right)^{-1}\sum_{\ell d_2d_3=k'}\mu(d_3)d_3^{-\overline{\alpha}}d_2^{-\overline{\alpha}-\gamma}\left(\prod_{p\mid \ell}\left(1-\frac{1}{p^{1+\overline{\alpha}}+\gamma}\right)\right)\sum_{d\mid \ell}\mu(d)
\\&=\prod_{p\mid k'}\left(1-\frac{1}{p^{1+\overline{\alpha}}}\right)^{-1}\sum_{ d_2d_3=k'}\mu(d_3)d_3^{-\overline{\alpha}}d_2^{-\overline{\alpha}-\gamma}
\\&=\frac{1}{{k'}^{\overline{\alpha}+\gamma}}\prod_{p\mid k'}\frac{1-p^{\gamma}}{1-p^{-1-\overline{\alpha}}}.
\end{align*}
Recall that 
\begin{align*}
	Z_{\alpha,\gamma,h,k}&=\frac{1}{h^{\alpha}k^{\gamma} }
	\frac{\zeta(1+\alpha+\gamma)}{\zeta(1+\alpha)}
	\prod_{p \mid k}  
	\frac{    1-p^{\gamma} }{ 1 - p^{-1-\alpha}}.
\end{align*}
Thus the contribution of $R_{s=1}$ to $\mathcal{M}$ is 
\begin{align*}
		\sum_{g\leqslant N}\sum_{\substack{h,k\leqslant N/g\\(h,k)=1}}\frac{y(gh)x(gk)}{ghk}\overline{\frac{d}{d\gamma}\frac{1}{2\pi}\int_{0}^{T} Z_{\overline{\alpha},\gamma,h,k}\, dt \Big|_{\gamma=0}    }.
\end{align*}
Now we consider the contribution of the second residue. We have 
	\begin{align*}
	R_{s=1-\overline{\alpha}-\gamma}=\frac{\zeta(1-\overline{\alpha}-\gamma)}{\zeta(1-\gamma)}G_1(1-\overline{\alpha}-\gamma,k,d,\alpha,\gamma)\frac{\left(\frac{kT}{2\pi h }\right)^{1-\overline{\alpha}-\gamma}}{1-\overline{\alpha}-\gamma}.
\end{align*}
By using the second formulation of $G(1-\overline{\alpha}-\gamma, k,d,\alpha,\gamma)$ in (\ref{another form of G for the second residue}), one can show that 
\begin{align*}
	\frac{\zeta(1-\overline{\alpha}-\gamma)}{h^{-\gamma}\zeta(1-\gamma)}
	{k'}^{\overline{\alpha}}\sum_{dk=k'}\frac{k^{1-\overline{\alpha}-\gamma}\mu(k)}{\phi(k)}G_1(1-\overline{\alpha}-\gamma, k,d,\alpha,\gamma)  =Z_{-\gamma,-\overline{\alpha},h,k'}.
\end{align*}
Since $\frac{1}{1-\overline{\alpha}-\gamma}\left(\frac{T}{2\pi}\right)^{1-\overline{\alpha}-\gamma}=\frac{1}{2\pi}\int_{0}^{T}\left(\frac{t}{2\pi}\right)^{-\overline{\alpha}-\gamma}\, dt$,  the contribution of $R_{s=1-\overline{\alpha}-\gamma}$ to $\mathcal{M}$ is 
\begin{align*}
	&\sum_{g\leqslant N}\sum_{\substack{h,k\leqslant N/g\\(h,k)=1}}\frac{y(gh)x(gk)}{ghk}
\overline{ \frac{d}{d\gamma }\frac{1}{2\pi}
 \int_{0}^{T}\left(\frac{t}{2\pi}\right)^{-\overline{\alpha}-\gamma} Z_{-\gamma,-\overline{\alpha},h,k}  \, dt \Big|_{\gamma=0}  }
\end{align*}
and this finishes the proof of Theorem \ref{Mthm}.
\qed

\section{Proof of Theorem \ref{ErrorUnconditional}. Bounding the error term $\mathcal{E} $ unconditionally} \label{sectionEuncond}
The proof of Theorem \ref{ErrorUnconditional} is very similar to that of \cite[Propositions 10 and 11]{HLZ} and thus we skip the details. 
As in \cite[Section 4.4]{HLZ}, we start with  
\begin{equation}
	\label{E1}
	\mathcal{E}_1(\gamma):=\sum_{2\leq q\leq \eta}  \,\,\chiqq \tau(\overline{\psi}) \sum_{k\leqslant N/q}\frac{b(qk)}{qk}\sum_{\substack{d\mid qk}} 
	\delta(q,qk,d,\psi) \sum_{m\leqslant qkT/2\pi d}a(dm)  	
	\psi \left(m\right) 
\end{equation}
and
\begin{equation}
	\label{E2}
	\mathcal{E}_2(\gamma):=\sum_{\eta\leq q\leq N}  \,\,\chiqq \tau(\overline{\psi}) \sum_{k\leqslant N/q}\frac{b(qk)}{qk}\sum_{\substack{d\mid qk}} 
	\delta(q,qk,d,\psi) \sum_{m\leqslant qkT/2\pi d}a(dm)  	
	\psi \left(m\right)
\end{equation}
so that $\mathcal{E}(\gamma)=\mathcal{E}_1(\gamma)+\mathcal{E}_2(\gamma)$ where $2\leqslant \eta\leqslant N$ is a parameter. First, we consider $\mathcal{E}_1(\gamma)$.
 \begin{lemma}
    For $\sigma\geqslant 1+\frac{1}{\log T}$, we have 
    \begin{align}\label{DSidentity}
\nonumber \sum_{m=1}^{\infty}\frac{a(dm) \psi (m)}{m^s}&=\sum_{d_1d_2d_3d_4=d}\sum_{(m_1,d_1d_2d_3)=1} \frac{\mu(m_1d_4)\psi(m_1)}{m_1^s(m_1d_4)^{\overline{\alpha}}}\sum_{(m_2,d_1d_2)=1} \frac{\psi(m_2)}{m_2^s(m_2d_3)^{\overline{\alpha}+\gamma}}
\nonumber\\&\nonumber\quad \quad \times\sum_{(m_3,d_1)=1} \frac{\psi(m_3)}{m_3^s}
\sum_{m_4\leq N} \frac{\psi(m_4)\overline{y(m_4d_1)}}{m_4^s(m_4d_1)^{\overline{\alpha}}}
\\&=  \frac{ L(s+\overline{\alpha}+\gamma,\psi)L(s, \psi) }{ L(s+\overline{\alpha}, \psi)}G_2(s,d,\alpha,\gamma)
\end{align}
where
\begin{align}
    \nonumber G_2(s,d,\alpha,\gamma,\psi)
    &:= j_{\psi}(s+\overline{\alpha}, d)^{-1} \sum_{d_1d_2d_3d_4=d} \mu(d_4)d_4^{-\overline{\alpha}}
d_3^{-\overline{\alpha}-\gamma}d_1^{-\overline{\alpha}}
j_{\psi}(s+\overline{\alpha}+\gamma,d_1 d_2) j_{\psi}(s,d_1)
\\&\quad \quad \quad \quad \quad \quad\quad \quad\quad \quad \quad \quad\times A_{\psi}(s+\overline{\alpha},d_1)
\end{align}
with
\begin{align*}
    j_{\psi}(z,n):=\prod_{p\mid n}\left(1-\frac{\psi(p)}{p^z}\right)
\end{align*}
and
\begin{align*}
    A_{\psi}(z,n):=\sum_{mn\leqslant N}\frac{\psi(m)\overline{y(mn)}}{m^z}.
\end{align*}
\end{lemma}
\begin{proof}
This identity is analogous to what is obtained in the proof of \cite[Lemma 12]{HLZ}.    By the definition of $a(\cdot)$ in (\ref{definition of a as convolution}), we have $a(n)=(f_1\ast f_2\ast f_3\ast f_4)(n) $ where $f_1(n)=\mu(n)n^{-\overline{\alpha}}, \, f_2(n)=n^{-\overline{\alpha}-\gamma}, f_3(n)=1,\, f_4(n)=\overline{y(n)}n^{-\alpha}$ for all $n\in\mathbb{N}$. By total multiplicativity of $\psi$ and Lemma \ref{convolutiondecomposition}, we have the desired result. 
\end{proof}
\begin{prop}
    Under the submultiplicativity assumptions (\ref{submultiplicative}) and the bounds (\ref{xyinftyunc}) for $x(n),y(n)$, if  $\eta\ll (\log T)^A$, then
    there exists a constant $C>0$ such that
\begin{align}
    \mathcal{E}_1 \ll T \exp(-C \sqrt{\log T}) \eta^{\frac{3}{2} +\varepsilon}.
\end{align}

\end{prop}
\begin{proof}
The estimate for partial sums of $a(dm)\psi(m)$  with $m\leqslant qkT/2\pi d$ is similar to the proof of \cite[Lemma 12]{HLZ} as  an application of Perron's formula, Lemma \ref{perron}. One only needs to be careful about the implied constants in the shifts $\alpha, \gamma \ll 1/\log T$ and the constant in $1-c/\log(qU)$ to make sure that the line of integration on the left remains inside the classical zero free region for the Dirichlet $L$-functions being considered as we performed in (\ref{J_integral}). As for the other essential ingredients in the proof, we use \cite[Lemma 6.6]{Ng} that for $\psi$ primitive and $kq\ll T$, we have 
	\begin{align}\label{bound_delta}
	|\delta(q,kq,d,\psi)|\ll \frac{(d,k)\log\log T}{\phi(k)\phi(q)},
	\end{align}
 and \cite[Lemma 6.7]{Ng} that for positive multiplicative functions $h$, we  have 
	\begin{align*}
\sum_{d\mid kq}\frac{(d,k)h(d)}{d}\ll (1\ast h)(k)||h(n)/n||_1.
\end{align*}
Thus we obtain the same bounds in  \cite[Lemma 12 and Proposition 10]{HLZ} that 
\begin{align*}
     \mathcal{E}_1(\gamma) \ll T \exp(-C \sqrt{\log T}) \eta^{\frac{3}{2} +\varepsilon}
\end{align*}
and the desired result follows by the Cauchy Integral Formula.
\end{proof}

The estimate above already gives the desired result in Theorem \ref{ErrorUnconditional} if $N$ is bounded by a power of $\log T$. If $N$ grows faster than any power of $\log T$, then we need to consider the contribution of large moduli in $\mathcal{E}_2(\gamma)$.

\begin{prop}
If $\eta\gg \log^A T$, then 
    \begin{align*}
    \mathcal{E}_2\ll T(\log T)^{-A}.
\end{align*}
\end{prop}
\begin{proof}
     Since the proof of our upper bound for the size of $\mathcal{E}_2(\gamma)$ is very similar to the proof of \cite[Proposition 11]{HLZ} of Heap, Li and Zhao inspired by the work \cite{BHB} of Bui and Heath-Brown, we only give an overview of the ideas being used. To drop the dependency of $d\mid kq$ on $q$, following Heap-Li-Zhao, \cite{HLZ}, we start with the change of variable $k=k'k_q$ where $(k',q)=1$ and $p\mid k_q$ implies $p\mid q$; and $d=d'k_q$ with $(d',k)=1$. Then by using $\delta(q,kq,d,\psi) \ll (\log\log T)^2 dk^{-1}q^{-1}$, $\left|\tau(\psi)\right|=q^{1/2}$, and the submultiplicativity assumptions in (\ref{submultiplicative}), the problem reduces to
\begin{align*}
    \mathcal{E}_2(\gamma)
	\ll 	(\log\log T)^2\sum_{\eta\leq q\leq N} \frac{|b(q)|}{q^{3/2}} \,\,\chiqq 
	\sum_{\substack{d\leq N/q}}\frac{|b(d)|}{d}\sum_{k\leq N/qd} \frac{|b(k)|}{k^2}
	\left|\sum_{m\leqslant qkT/2\pi }a(dm)  	
	\psi \left(m\right) \right|.
\end{align*}
Then considering dyadic intervals for $k,q$ and $d$; and taking the maximum contribution of such intervals, we see that there exist $K,Q$ and $D$ with  $KQD\ll N$   such that 
\begin{align*}
	\mathcal{E}_2(\gamma)
	\ll  (\log T)^3	(\log\log T)^2  \sum_{q\sim Q} \frac{|b(q)|}{q^{3/2}} 
	\sum_{d\sim D}\frac{|b(d)|}{d}\sum_{k\sim K} \frac{|b(k)|}{k^2}
	\,\,\chiqq 	\left|\sum_{m\leqslant qkT/2\pi }a(dm)  	
	\psi \left(m\right) \right| 
\end{align*}
and by using the bounds in (\ref{xyinftyunc}), we have 
\begin{align*}
	\mathcal{E}_2(\gamma)
	\ll  \frac{(\log T)^{C}}{K Q^{3/2-\varepsilon}} 
	\sum_{d\sim D}\frac{|b(d)|}{d} \sum_{q\sim Q} 
	\,\,\chiqq  \max_{x\leq 2KQT}	\left|\sum_{m\leqslant x}a(dm)  	
	\psi \left(m\right) \right|
\end{align*}
for some constant $C>0$ where the sums over $d\sim D, q\sim Q$ run over $D\leqslant d\leqslant 2D$ and $Q\leqslant q\leqslant 2Q$.  Then by using Heath-Brown's identity for decomposing the M\"{o}bius function, we can decompose $\mu(n)n^{-\alpha}$ and obtain a decomposition of $a(dm)$ into $2J+2$ functions $f_i$ where $J\in\mathbb{N}$ sufficiently large and $ f_{2J} (n) = 1,\, f_{2J+1}(n)=n^{-\overline{\alpha}-\gamma},\, f_{2J+2} = \overline{y(n)}n^{-\overline{\alpha}}$. By Lemma \ref{convolutiondecomposition}, we obtain a further decomposition in the form
$$a(md)= (f_1\ast\dots\ast f_{2J+2})(md)=\sum_{d_1\dots d_{2J+2}=d}  (g_1\ast\dots\ast g_{2J+2})(m)$$
with 
\begin{align*}
	g_i(m)=\begin{cases}
		f_i(md_i) \,\,\,&\text{if}\,(m,D_i)=1\,\text{where}\,D_i=d_1\dots d_{i-1},\\
		0\,\,\,\,\,&\text{otherwise}.
	\end{cases}
\end{align*}
Thus the estimates of Type I sums where the main ingredients are partial summation and the P\'{o}lya-Vinogradov inequality and of Type II sums where large sieve inequality is used  arising from such a decomposition of $2J+2$ functions follow along the lines of the proof of \cite[Proposition 11]{HLZ}. Thus we obtain the same upper bound
\begin{align*}
    \mathcal{E}_2(\gamma)\ll N^{4/3}T^{1/3+\varepsilon}+(\log T)^{C'}(N T^{1/2+\varepsilon}+T\eta^{-1/2+\varepsilon}+N^{1/3}T^{5/6+\varepsilon}).
\end{align*}
If $N$ grows faster than any fixed power of $\log T$, then we can take $\eta=\log^A(T)$ and thus $\mathcal{E}_2\ll T/\log^AT$ by the Cauchy Integral Formula.
\end{proof}

  Combining the estimates for $\mathcal{E}_1(\gamma)$ and $\mathcal{E}_2(\gamma)$ and an application of the Cauchy Integral Formula finishes the proof of the bound $\mathcal{E}\ll T/\log^AT$ in the unconditional case.

\section{Proof of Theorem \ref{EGRH}: Bounding $\mathcal{E}$ assuming GRH($\Theta$).} \label{sectionEGRH}
Recall 
\begin{align*}
    \mathcal{E}(\gamma):=\sum_{2\leq q\leq N}  \,\,\chiqq \tau(\overline{\psi}) \sum_{k\leqslant N/q}\frac{b(qk)}{qk}\sum_{\substack{d\mid qk}} 
	\delta(q,qk,d,\psi) \sum_{m\leqslant qkT/2\pi d}a(dm)  	
	\psi \left(m\right) 
\end{align*}
Assume the GRH($\Theta$) conjecture as stated in Conjecture \ref{GRHtheta} for some  $\Theta\in [1/2,1)$.
Let 
\begin{align*}
    x=\frac{qkT}{2\pi d}.
\end{align*}
 Define 
\begin{align*}
   F_2(s):= \frac{ L(s+\overline{\alpha}+\gamma,\psi)L(s, \psi) }{ L(s+\overline{\alpha}, \psi)}G_2(s,d,\alpha,\gamma)\frac{x^s}{s}
\end{align*}
where we recall 
\begin{align}
    \nonumber G_2(s,d,\alpha,\gamma,\psi)
    &= j_{\psi}(s+\overline{\alpha}, d)^{-1} \sum_{d_1d_2d_3d_4=d} \mu(d_4)d_4^{-\overline{\alpha}}
d_3^{-\overline{\alpha}-\gamma}d_1^{-\overline{\alpha}}
j_{\psi}(s+\overline{\alpha}+\gamma,d_1 d_2) j_{\psi}(s,d_1)
\\&\quad \quad \quad \quad \quad \quad\quad \quad\quad \quad \quad \quad\times A_{\psi}(s+\overline{\alpha},d_1)
\end{align}
with
\begin{align*}
    j_{\psi}(z,n)=\prod_{p\mid n}\left(1-\frac{\psi(p)}{p^z}\right)
\end{align*}
and
\begin{align*}
    A_{\psi}(z,n)=\sum_{mn\leqslant N}\frac{\psi(m)\overline{y(mn)}}{m^z}.
\end{align*}
Note that for $\sigma\geqslant 1/2$, we have 
\begin{align*}
    j_{\psi}(s,n)^{-1}=\prod_{p\mid n}\sum_{k=0}^{\infty}\left(\frac{\psi(p)}{p^s}\right)^k\ll \prod_{p\mid n}\left(1+\frac{C}{p^{\sigma}}\right)\ll \tau(n)\ll T^\varepsilon
\end{align*}
for $n\ll N^2\ll T^{2\vartheta}$. Similarly, $j_{\psi}(s,n)\ll T^{\varepsilon}$ for $n\ll N^2$ and $\sigma\geqslant 1/2$. Thus, for $\sigma\geqslant \frac{1}{2}+\varepsilon$, we have 
\begin{align*}
    G_2(s,d,\alpha,\gamma,\psi)&\ll T^{\varepsilon}\sum_{d_1d_2d_3d_4=d}\sum_{md_1\leqslant N}\frac{\left|y(m)y(d_1)\right|}{m^{\sigma}}
    \\&= T^{\varepsilon}\sum_{m\leqslant N}\frac{\left|y(m)\right|}{m^{\sigma}}\sum_{\substack{d_1d_2d_3d_4=d\\ d_1\leqslant N/m}}\left|y(d_1) \right|
    \\&\ll T^{\varepsilon} \left(\sum_{\ell\mid d}\left|y(\ell)\right|\right) \left\|\frac{y(n)}{n^\sigma}\right\|_1
\end{align*}
since $\alpha,\gamma\ll 1/\log T$ and $\tau_3(d)\ll T^{\varepsilon}$.

Let $\kappa=1+\frac{1}{\log T}$, $2\leqslant U\leqslant T^{2} $ and $H\leqslant \left(\frac{x}{N}\right)^{1-\varepsilon}$. By Lemma \ref{perron}, we have 
\begin{align}\label{perron_2}
		\nonumber \sum_{m\leqslant x}a(dm) \psi(m)&=\frac{1}{2\pi i}\int_{\kappa-iU}^{\kappa+iU}	F_2(s)\, ds+O\left(\sum_{x-\frac{x}{H}\leqslant n\leqslant x+\frac{x}{H}}\left|a(dn)\right|\right)
  \\&\quad \quad \quad +O\left( \frac{x^{\kappa}H}{U}\sum_{n=1}^{\infty}\left|a(dn)\right|n^{-\kappa}   \right).
\end{align}
Since $a(dn)\ll(dn)^{4/15\log T} \left(\tau_3\ast \left|y\right|\right) (dn)\ll(dn)^{4/15\log T} \left(\tau_3\ast \left|y\right|\right) (d)  \left(\tau_3\ast \left|y\right|\right) (n) $ and $d\leqslant T$, we have 
\begin{align*}
    \sum_{n=1}^{\infty}\left|a(dn)\right|n^{-\kappa} \ll \sum_{n=1}^{\infty} \frac{(dn)^{4/15\log T} \left(\tau_3\ast \left|y\right|\right) (d)  \left(\tau_3\ast \left|y\right|\right) (n) }{n^{1+\frac{1}{\log T}}}\ll T^\varepsilon \left(\sum_{\ell\mid d}\left|y(\ell)\right|\right)\left\|\frac{y(n)}{n}\right\|_1
\end{align*}
and thus the second error term in (\ref{perron_2}) is bounded by 
\begin{align}\label{second_error}
    T^\varepsilon \frac{xH}{U} \left(\sum_{\ell\mid d}\left|y(\ell)\right|\right)\left\|\frac{y(n)}{n}\right\|_1.
\end{align}
Since $x\leqslant T^2$, the first error term in (\ref{perron_2}) is bounded by 
\begin{align}\label{first_error}
  T^\varepsilon    \left(\sum_{\ell\mid d}\left|y(\ell)\right|\right)\sum_{n\leqslant N}\left|y(n)\right|\sum_{x-\frac{x}{H}\leqslant mn\leqslant x+\frac{x}{H}}\tau_{3}(m)\ll T^\varepsilon\frac{x}{H} \left(\sum_{\ell\mid d}\left|y(\ell)\right|\right)\left\|\frac{y(n)}{n}\right\|_1
\end{align}
by Shiu's divisor sum bound, \cite[Theorem 2]{S}, since $H\leqslant \left(\frac{x}{N}\right)^{1-\varepsilon}$. 

Let $\sigma_2:=\Theta+\varepsilon$ and define 
\begin{align}\label{J_2_integral}
	J_2(\sigma_2):=\frac{1}{2\pi i}\left( \int_{\kappa+iU}^{\sigma_2+iU}+\int_{\sigma_2+iU}^{\sigma_2-iU}+\int_{\sigma_2-iU}^{\kappa+iU}  \right)	\frac{d}{d\gamma }  F_2(s) \Bigr|_{\gamma=0}	\, ds.
\end{align}
and note that 
\begin{align*}
	\frac{d}{d\gamma }  F_2(s)  \Bigr|_{\gamma=0}=\left(\frac{L'}{L}(s+\overline{\alpha},\psi)G_2(s,d,\alpha,0)+ G_2'(s,d,\alpha, 0)\right)L(s,\psi)\frac{x^s}{s}
\end{align*}
where $G_2'(s,d,\alpha, 0)=\frac{d}{d\gamma}G_2(s,d,\alpha, \gamma)|_{\gamma=0}\ll T^{\varepsilon} \left(\sum_{\ell\mid d}\left|y(\ell)\right|\right) \left\|\frac{y(n)}{n^\sigma}\right\|_1$ by the Cauchy Integral Formula.  Since
\begin{align*}
    \frac{L'}{L}(s+\overline{\alpha},\psi),\, L(s,\psi)\ll (qT)^{\varepsilon}\ll T^{2\varepsilon}
\end{align*}
along the path of integration considered in (\ref{J_2_integral}), we have 
\begin{align*}
    J_2(\sigma_2)\ll T^{\varepsilon} \left(\sum_{\ell\mid d}\left|y(\ell)\right|\right) \left\|\frac{y(n)}{n^\Theta}\right\|_1\left(  \frac{x}{U}+x^{\Theta}   \right).
\end{align*}
Thus the total contribution of the error terms are bounded by 
\begin{align}
   \ll T^{\varepsilon} \left(\sum_{\ell\mid d}\left|y(\ell)\right|\right) \left\|\frac{y(n)}{n^\Theta}\right\|_1\left(x^{\Theta}+\frac{xH}{U} +\frac{x}{H}           \right)
\end{align}
uniformly in $\left|\alpha\right|\leqslant 1/15\log T$.
We take $U=x^{2-2\Theta}$ and $H=x^{1-\Theta-\varepsilon}$. Note that $H\leqslant \left(\frac{x}{N} \right)^{1-\varepsilon}$ since $N^{1-\varepsilon}\leqslant T^{\vartheta(1-\varepsilon)} \leqslant \left(\frac{qkT}{2\pi d}\right)^{\Theta}$ for $d\mid qk$. Thus 
\begin{align}\label{Epsilon_bound_1}
\mathcal{E}\ll T^{\Theta+\varepsilon}  \left\|\frac{y(n)}{n^\Theta}\right\|_1\sum_{2\leq q\leq N}  \,\,\chiqq \sqrt{q} \sum_{k\leqslant N/q}\frac{\left|x(qk)\right|}{qk}\sum_{\substack{d\mid qk}}  \left(\sum_{\ell\mid d}\left|y(\ell)\right|\right)
	\frac{(d,k)}{\phi(k)\phi(q)}\left( \frac{qk}{d}\right)^\Theta
\end{align}
by (\ref{bound_delta}).
For $d\mid kq\ll T$, we have $\frac{(d,k)}{\phi(k)\phi(q)}\left( \frac{qk}{d}\right)^\Theta\ll  T^\varepsilon \frac{d}{kq}\left( \frac{qk}{d}\right)^\Theta= T^{\varepsilon} \left(\frac{d}{qk}\right)^{1-\Theta}$. Thus 
\begin{align*}
    \mathcal{E}&\ll T^{\Theta+\varepsilon}  \left\|\frac{y(n)}{n^\Theta}\right\|_1\sum_{2\leqslant q\leqslant N} \sum_{k\leqslant N/q}\frac{q^{3/2}\left|x(qk)\right|}{(qk)^{2-\Theta}}\sum_{d\mid qk}d^{1-\Theta} \sum_{\ell \mid d}\left|y(\ell)\right|
    \\&\ll T^{\Theta+\varepsilon}  \left\|\frac{y(n)}{n^\Theta}\right\|_1\sum_{2\leqslant q\leqslant N}\sum_{k\leqslant N/q}(qk)^{1/2}x(qk)(1\ast \left| y\right|)(qk)
    \\&\ll T^{\Theta+\varepsilon}  \left\|\frac{y(n)}{n^\Theta}\right\|_1\left\|n^{1/2}x(n)(1\ast\left| y\right|)(n)\right\|_1.
\end{align*}
\qed

\section{Proofs of Corollary \ref{corollary} and \ref{corollary_lower_bound}} \label{sectioncorollaries}
In this final section, we establish our corollaries. 
\begin{proof}[Proof of Corollary \ref{corollary}]
Let $m\geqslant 1$ be fixed.	By the Cauchy Integral Formula, we have 
	\begin{align*}
		\sum_{0 < \gamma \leq T} \zeta^{(m)}(\rho)X(\rho) Y(1\!-\! \rho)=\frac{m!}{2\pi i}\int_{\mathcal{C}}\frac{S(\alpha,T,X,Y)}{\alpha^{m+1}}\, d\alpha
	\end{align*}
where $\mathcal{C}$ is a positively oriented circle with radius $\ll \frac{1}{\log T}$ centered at the origin. Note that in (\ref{dmformula}), the term $\mathcal{E}$ is analytic at $\alpha=0$ since all the other terms in (\ref{dmformula}) are analytic therein. The contributions of the derivatives of the terms on the right-hand side of (\ref{dmformula}) are calculated in Lemmata  \ref{lemma_derivative_1}-\ref{lemma_F_for_k>1}. Since $\mathcal{E}$ is analytic at $\alpha=0$ and the upper bounds for the size of $\mathcal{E}$ in Theorem \ref{mainthm} are uniform in $\alpha\ll \frac{1}{\log T}$, the contribution of the term $\mathcal{E}$ to the integral above is  $\ll  m!\log^{m}T\max_{\alpha\in\mathcal{C}}\left|\mathcal{E}\right|$. Since $m$ is fixed and the bounds in (i) and (ii) in Theorem \ref{mainthm} hold for all $A>0$ and for all $\varepsilon>0$, respectively,  we obtain the desired result.

\end{proof}

\begin{proof}[Proof of Corollary  \ref{corollary_lower_bound}]
We follow the approach in \cite{MN2}. 	For $0<\vartheta<\frac{1}{2}$ and sufficiently large $T$, let $N=\xi^k= T^{\vartheta}$. Define 
$\mathcal{C}_{\xi}(s):=\sum_{n\leqslant \xi}\frac{1}{n^s}$, 
\begin{align*}
		\Sigma_1:= \sum_{0 < \gamma < T}\zeta^{(m)}(\rho)\mathcal{C}_{\xi}(\rho)^{k-1}\overline{\mathcal{C}_{\xi}(\rho)}^{k}, \, 
	\text{ and }
	\Sigma_{2}:=\sum_{0 < \gamma < T}\left|\mathcal{C}_{\xi}(\rho)\right|^{2k}.
\end{align*}

By H\"{o}lder's  inequality, 
\begin{align*}
	\left|\Sigma_{1}\right|^{2k}\leqslant \left(\sum_{0 < \gamma < T}\left|\zeta^{(m)}(\rho)\right|^{2k}\right)\Sigma_{2}^{2k-1}
\end{align*}
and thus
$\sum_{0 < \gamma < T}\left|\zeta^{(m)}(\rho)\right|^{2k}\geqslant \frac{\left|\Sigma_{1}\right|^{2k}}{\Sigma_{2}^{2k-1}}$.
First we find a lower bound for $\left|\Sigma_{1}\right|^{2k}$. By the assumption of the Riemann Hypothesis, we have 
\begin{align*}
	\Sigma_1=\sum_{0 < \gamma < T}\zeta^{(m)}(\rho)\mathcal{C}_{\xi}(\rho)^{k-1}\mathcal{C}_{\xi}(1-\rho)^{k}.
\end{align*}
For $k\geqslant 1$, define 
$$\tau_{k}(n,\xi):=\sum_{\substack{d_1d_2\dots d_{k}=n\\d_1,d_2\dots, d_k\leqslant \xi}}1.$$
In Corollary \ref{corollary}, we take 
\begin{align*}
	x(n):=\begin{cases}
		\tau _{k-1}(n,\xi) & \text{ if } k>1,\, n\leqslant \xi^{k-1},
		\\ 1  & \text{ if } k=n=1,
		\\ 0  & \text{ otherwise.}   
	\end{cases}
\text{ and }\,
 	y(n):=\begin{cases}
 		\tau _{k}(n,\xi) & \text{ if } n\leqslant \xi^{k}=N,
 		\\ 0  & \text{ otherwise.}   
 	\end{cases}
 \end{align*}
By Corollary \ref{corollary}, we have 
\begin{align}\label{Sigma_1}
\Sigma_1=\mathcal{T}_1
	+\mathcal{T}_2
\end{align}
where 
\begin{align*}
	\mathcal{T}_1&:=\frac{(-1)^{m+1}}{m+1}\frac{T}{2\pi }
	\sum_{g\leqslant N}\sum_{\substack{h\leqslant N/g}}\frac{y(gh)x(g)}{gh} \left(  \mathcal{P}_{m+1}(\Lo)  - \mathcal{Q}_{m+1}(\log h) \right)
	\\&\quad +(-1)^{m+1}\frac{T}{2\pi}\sum_{n\leqslant N}\frac{\left(\Lambda\ast  \log^m \ast x\right) (n)y(n)}{n} 
	+(-1)^{m+1}\frac{T}{2\pi }\sum_{g\leqslant N}\sum_{\substack{h,r\leqslant N/g\\r\geqslant2\\(h,r)=1}}\frac{y(gh)x(gr)}{ghr}\mathcal{A}_m(h,r),
\end{align*}
\begin{align*}
	\mathcal{T}_2&:=\frac{T}{2\pi }\sum_{g\leqslant N}\sum_{\substack{h,r\leqslant N/g\\r\geqslant2\\(h,r)=1}}\frac{y(gh)x(gr)}{ghr}\mathcal{B}_m(r,T)
	+(-1)^m\frac{T}{2\pi}\log(\tfrac{T}{2\pi e}) 
	\sum_{n\leqslant N}\frac{\left( \log^m * x\right) (n)y(n)}{n}+\mathcal{E}.
\end{align*}
Since $\mathcal{P}_{m+1}(\cdot)$ and $\mathcal{Q}_{m+1}(\cdot)$ are monic polynomials of degree $m+1$, we have 
\begin{align*}
	\mathcal{P}_{m+1}(\Lo)\geqslant (1+o(1))\Lo^{m+1}\geqslant(1+o(1))\log^{m+1}T 
\end{align*}
and  $\mathcal{Q}_{m+1}(\log h)\leqslant (1+o(1))\log^{m+1}h$
for sufficiently large $T$ and $h$. For those $h$ under consideration, we have $h\leqslant N\leqslant T^{\vartheta}$ and thus 
\begin{align*}
		\mathcal{Q}_{m+1}(\log h)\leqslant (1+o(1))\left(\vartheta \log T\right)^{m+1}=(1+o(1))\vartheta^{m+1}(\log T)^{m+1}.
\end{align*}
Thus $\mathcal{P}_{m+1}(\Lo)-	\mathcal{Q}_{m+1}(\log h)\geqslant (1-\vartheta^{m+1}+o(1))(\log T)^{m+1}$ and 
$\mathcal{P}_{m+1}(\Lo)-	\mathcal{Q}_{m+1}(\log h)$ is positive for sufficiently large $T$. It follows that  
\begin{align*}
\left|\mathcal{T}_1\right|\geqslant \frac{1-\vartheta^{m+1}+o(1)}{2\pi (m+1)}T(\log T)^{m+1}\sum_{g\leqslant N}\sum_{\substack{h\leqslant N/g}}\frac{y(gh)x(g)}{gh}	
\end{align*}
by the positivity of $\mathcal{P}_{m+1}(\Lo)-	\mathcal{Q}_{m+1}(\log h)$ and the nonnegativity of $\left(\Lambda\ast  \log^m \ast x\right) (n)y(n)$ and $\mathcal{A}_m(h,r)$ (recall definition  \eqref{definition_mathcal_A}) whose contributions are dropped. 

Now, we find an upper bound for $\left|\mathcal{T}_2\right|$. 
Define 
\begin{align*}
	\mathcal{T}_{2,1}:=\frac{T}{2\pi }\sum_{g\leqslant N}\sum_{\substack{h,r\leqslant N/g\\r\geqslant2\\(h,r)=1}}\frac{y(gh)x(gr)}{ghr}\mathcal{B}_m(r,T), \, 
	\mathcal{T}_{2,2}:=(-1)^m\frac{T}{2\pi}\log\left(\tfrac{T}{2\pi e}\right) 
	\sum_{n\leqslant N}\frac{\left( \log^m * x\right) (n)y(n)}{n}
\end{align*}
so that $\mathcal{T}_{2}=\mathcal{T}_{2,1}+\mathcal{T}_{2,2}+\mathcal{E}$.
Recall that, for $r>1$, $\mathcal{B}_m(r,T)=0$ if $\omega(r) \ge m+2$ and if $1 \le \omega(r) \le m+1$
\begin{align*}
	\mathcal{B}_m(r,T)=m!\frac{r(-1)^{\omega(r)+1}}{\varphi(r)}
	\sum_{\substack{u_1+u_2\leqslant m\\u_1\geqslant \omega(r)-1\\u_2\geqslant 0}} \mathcal{G}_{u_1+1}(r)\frac{(-1)^{u_2}}{u_2!}\log^{u_2}\left(\frac{T}{2\pi r}\right)\sum_{j=-1}^{m-u_1-u_2-1}\tilde{\gamma}_{j}.
\end{align*}
where  we recall $\mathcal{G}_u(r)$ is defined by \eqref{definition_mathcal_G}.
From the definition \eqref{definition_mathcal_G}, 
 it follows that 
for $1\leqslant\omega(r)\leqslant m+1$ and $\omega(r)-1\leqslant u_1\leqslant m$, we have  
\begin{align*}
	\left|	\mathcal{G}_{u_1+1}(r)\right|&\leqslant \sum_{\substack{\ell_1+\ell_2+\dots+\ell_{\omega(r)}=u_1+1\\\ell_1,\dots, \ell_{\omega(r)}\geqslant 0}}\left(\frac{\log^{\ell_1}(p_1)}{\ell_1!}\frac{\log^{\ell_2}(p_2)}{\ell_2!}\dots\frac{\log^{\ell_{\omega(r)}}(p_{{\omega(r)}})}{\ell_{\omega(r)}!}\right)
	\\&= \frac{1}{(u_1+1)!}\left(\sum_{p\mid r}\log p\right)^{u_1+1}
	\leqslant \frac{\log^{u_1+1}\tilde{r}}{(u_1+1)!}
\end{align*}
where $\tilde{r}=\prod_{j=1}^{\omega(k)}p_j$ is the square-free part of $r$. Since $1\leqslant \omega(r)\leqslant m+1$ for those $r$ under consideration, we have 
$\frac{r}{\varphi(r)}\leqslant \prod_{j=1}^{m+1}(1-\frac{1}{q_j})^{-1}$
where 
$q_j$ denotes the $j^{\text{th}}$ prime number. Thus 
\begin{align*}
	\left|\mathcal{B}_m(r,T)\right|\leqslant C_m	\mathbbm{1}_{1\leqslant \omega(r)\leqslant m+1}\sum_{\substack{u_1+u_2\leqslant m\\u_1\geqslant \omega(r)-1\\u_2\geqslant 0}}\frac{\log^{u_1+1}\tilde{r}}{(u_1+1)!}\frac{\log^{u_2}T}{u_2!}
\end{align*}
for some constant $C_m$ depending only on $m$. By \cite[Lemma 3.1]{MN2}, we have  $x(gr)\leqslant x(g)x(r)$ and thus we have 
\begin{align*}
\mathcal{T}_{2,1} \ll T\sum_{g\leqslant N}\sum_{\substack{h\leqslant N/g}}\frac{y(gh)x(g)}{gh}			\sum_{\substack{u_1+u_2\leqslant m\\u_1,u_2\geqslant 0}}\frac{1}{(u_1+1)!}\frac{\log^{u_2}T}{u_2!}\sum_{\substack{2\leqslant r \leqslant \frac{N}{gh}\\(h,r)=1\\1\leqslant \omega(r)\leqslant u_1+1}}\frac{x(r)}{r}\log^{u_1+1}\tilde{r}
\end{align*}
where the implied constant depends on $m$. If $k=1$, then the inner sum above is void since in this case $x(1)=1$ and $x(r)=0$ for $r>1$. Thus we may assume that $k\geqslant 2$. We have
\begin{align*}
\sum_{\substack{2\leqslant r \leqslant \frac{N}{gh}\\(h,r)=1\\1\leqslant \omega(r)\leqslant u_1+1}}\frac{x(r)}{r}\log^{u_1+1}\tilde{r}&\leqslant \sum_{j=1}^{u_1+1}\sum_{\substack{p_1,p_2\dots, p_j\\p_1^{a_1}\dots p_j^{a_j}\leqslant N}}\frac{\tau_{k-1}\left(p_1^{a_1}\dots p_j^{a_j}\right)}{p_1^{a_1}\dots p_j^{a_j}}\left(\log p_1+\log p_2+\dots +\log p_j\right)^{u_1+1}	
\end{align*}  
where the sum over $p_1,\dots, p_j$ runs over distinct prime numbers. By Lemma \ref{lemma_for__moment_corollary}, the right-hand side above is 
$\ll \log^{u_1+1} N$
where the implied constant depends on $k$ and $u_1$. Thus 
\begin{align*}
\mathcal{T}_{2,1}&\ll T\sum_{g\leqslant N}\sum_{\substack{h\leqslant N/g}}\frac{y(gh)x(g)}{gh}			\sum_{\substack{u_1+u_2\leqslant m\\u_1,u_2\geqslant 0}}\frac{\log^{u_1+1} N}{(u_1+1)!}\frac{\log^{u_2}T}{u_2!}
\\&\ll	\left(\sum_{\substack{u_1+u_2\leqslant m\\u_1,u_2\geqslant 0}}\frac{\vartheta^{u_1+1}}{(u_1+1)!}\frac{1}{u_2!}\right)T\left(\log T\right)^{m+1} \sum_{g\leqslant N}\sum_{\substack{h\leqslant N/g}}\frac{y(gh)x(g)}{gh}	
\end{align*}
where the implied constant depends  on $k$ and $m$, but not on $\vartheta$. Now, we consider $\mathcal{T}_{2,2}$. We have 
\begin{align*}
		& \sum_{n\leqslant N}\frac{\left( \log^m *x\right) (n)y(n)}{n}=\sum_{n\leqslant N}\frac{y(n)}{n}\sum_{g\mid n}\log^m\left(n/g\right)x(g)
		\\&\leqslant \log^mN \sum_{n\leqslant N}\frac{y(n)}{n}\sum_{g\mid n}x(g)
		\leqslant \vartheta^m\left(\log T\right)^m\sum_{g\leqslant N}\sum_{h\leqslant N/g}\frac{y(gh)x(g)}{gh}.
\end{align*}
Thus
$
	\mathcal{T}_{2,2}\ll \vartheta^mT\left(\log T\right)^{m+1}\sum_{g\leqslant N}\sum_{h\leqslant N/g}\frac{y(gh)x(g)}{gh}
$
where the implied constant is absolute. By using the estimates for $\mathcal{T}_1$ and $\mathcal{T}_2$ and the fact that $\mathcal{E}=o(T)$, we have 
\begin{align*}
\left|\Sigma_1\right|&\geqslant \left(  \frac{1-\vartheta^{m+1}+o(1)}{2\pi (m+1)} -C_{k,m}\sum_{\substack{u_1+u_2\leqslant m\\u_1,u_2\geqslant 0}}\frac{\vartheta^{u_1+1}}{(u_1+1)!}\frac{1}{u_2!}-\vartheta^m -o(1)   \right)
\\&\quad \quad \quad \times	T\left(\log T\right)^{m+1}\sum_{g\leqslant N}\sum_{h\leqslant N/g}\frac{y(gh)x(g)}{gh}
\end{align*}
where $C_{k,m}$ is a positive constant depending only on $k$ and $m$. By choosing $\vartheta$ sufficiently small, the expression inside the brackets above is $\geqslant 1/(4\pi (m+1))$ and thus 
\begin{align*}
	\Sigma_1\gg T\left(\log T\right)^{m+1}\sum_{g\leqslant N}\sum_{h\leqslant N/g}\frac{y(gh)x(g)}{gh}
	\gg (\log T)^{k^2},
\end{align*}
where the last bound was established in  \cite[Page 24]{MN2}.
In \cite[Section 4]{MN2}, it is proved that 
$\Sigma_2\ll T\left(\log T\right)^{k^2+1}$ and thus we deduce
\begin{align*}
	\sum_{0 < \gamma < T}\left|\zeta^{(m)}(\rho)\right|^{2k}\geqslant \frac{\left|\Sigma_{1}\right|^{2k}}{\Sigma_{2}^{2k-1}}\gg \frac{T^{2k}\left(\log T\right)^{2k\left(k^2+m+1\right)}}{T^{2k-1}\left(\log T\right)^{(2k-1)\left(k^2+1\right)}}=T(\log T)^{k^2+2km+1}.
\end{align*}
\end{proof}


\begin{thebibliography}{99}
\bibitem{Apostol} T. M. Apostol, \textit{Introduction to Analytic Number Theory}, Undergrad. Texts Math.
Springer-Verlag, New York-Heidelberg, 1976. xii+338 pp.





\bibitem{B}
A. R.  Booker, {\it Poles of Artin L-functions and the strong Artin conjecture}. Ann. of Math. (2) 158 (2003), no. 3, 1089–1098. 

\bibitem{B2}
A. R.  Booker,  {\it Simple zeros of degree 2 L-functions}. J. Eur. Math. Soc. (JEMS) 18 (2016), no. 4, 813–823. 


\bibitem{BCK}
A.R. Booker, P.J. Cho, M. Kim, 
{\it Simple zeros of automorphic L-functions}. 
Compos. Math. 155 (2019), no. 6, 1224–1243. 

\bibitem{BMN}
A.R. Booker, M.B. Milinovich, N. Ng, {\it Quantitative estimates for simple zeros of L-functions}. Mathematika 65 (2019), no. 2, 375–399. 

\bibitem{BHB}
 H.M. Bui, D.R. Heath-Brown, {\it On simple zeros of the Riemann zeta-function}. Bull. Lond. Math. Soc. 45 (2013), no. 5, 953–961.

\bibitem{BFM}
 H.M. Bui, A. Florea, M.B. Milinovich, {\it
Negative discrete moments of the derivative of the Riemann zeta function}, preprint, \url{https://arxiv.org/abs/2310.03949}.

 
 \bibitem{BMN2}
 H.M. Bui, M.B. Milinovich, N.C. Ng,  {\it A note on the gaps between consecutive zeros of the Riemann zeta-function}. Proc. Amer. Math. Soc. 138 (2010), no. 12, 4167–4175.

 \bibitem{C}
 M. W. Coffey,  {\it Relations and positivity results for the derivatives of the Riemann $\xi$ function}. J. Comput. Appl. Math. 166 (2004), no. 2, 525-534.
 
 \bibitem{CG}
  J. B. Conrey, A. Ghosh,  {\it Simple zeros of the Ramanujan $\tau$-Dirichlet series}. Invent. Math. 94 (1988), no. 2, 403–419.
 
 \bibitem{CGG1} J. B. Conrey, A. Ghosh and S. M. Gonek, M. {\it Simple zeros of the zeta function of a quadratic number field. I}. Invent. Math. 86 (1986), no. 3, 563–576. 
  
  \bibitem{CGG2} J. B. Conrey, A. Ghosh and S. M. Gonek, {\it Large gaps between zeros of the zeta-function}. Mathematika 33 (1986), no. 2, 212–238 (1987). 
 
 
\bibitem{CGG3} J. B. Conrey, A. Ghosh and S. M. Gonek, {\it Simple zeros of the zeta-function of a quadratic number field.} II. Analytic number theory and Diophantine problems (Stillwater, OK, 1984), 87–114, Progr. Math., 70, Birkhäuser Boston, Boston, MA, 1987.
 
 
\bibitem{CGG4}  J. B. Conrey, A. Ghosh and S. M. Gonek, {\it Simple zeros of zeta functions}. Colloque de Théorie Analytique des Nombres "Jean Coquet'' (Marseille, 1985), 77–83, Publ. Math. Orsay, 88-02, Univ. Paris XI, Orsay, 1988. 


 
\bibitem{CGG0} J. B. Conrey, A. Ghosh and S. M. Gonek, {\it Mean values of the Riemann zeta-function with
application to the distribution of zeros}, Number theory, trace formulas and discrete groups
(Academic Press, Boston, 1989), pp. 185-199.
 
 \bibitem{CGG}
 J.B. Conrey,  A. Ghosh, S.M. Gonek, 
{\it Simple zeros of the Riemann zeta-function}.
Proc. London Math. Soc. (3) 76 (1998), no. 3, 497–522.



\bibitem{Da} H. Davenport, {\it Multiplicative number theory}, 3rd edn (revised by H.L. Montgomery), Graduate texts in mathematics,  74, (Springer-Verlag, New York, 1980). 

\bibitem{Fa}
A. de Faveri, {\it Simple zeros of GL(2) L-functions}, preprint, \url{https://arxiv.org/abs/2109.15311}. 

\bibitem{F1}
A. Fujii,  {\it 
On a conjecture of Shanks}. 
Proc. Japan Acad. Ser. A Math. Sci. 70 (1994), no. 4, 109-114. 

\bibitem{F2}
A. Fujii,  {\it 
On the distribution of values of the derivative of the Riemann zeta function at its zeros. I}. (English summary) 
Tr. Mat. Inst. Steklova 276 (2012), Teoriya Chisel, Algebra i Analiz, 57–82 ISBN: 5-7846-0121-0; 978-5-7846-0121-6 ; reprinted in 
Proc. Steklov Inst. Math. 276 (2012), no. 1, 51–76.

\bibitem{G}
P. Gao, {\it Sharp lower bounds of moments of $\zeta'(\rho)$}, preprint, 
\url{https://arxiv.org/abs/2106.03057}.

 \bibitem{GZ}
P. Gao, L. Zhao, 
{\it Lower bounds for negative moments of $\zeta'(\rho)$}. 
Mathematika 69 (2023), no. 4, 1081–1103. 

 \bibitem{G1}  S. M. Gonek,  {\it Mean values of the Riemann zeta-function and its derivatives}.  Invent. Math. 75 (1984): 123--141. 
 
 \bibitem{G3}
 S. M.  Gonek,    {\it On negative moments of the Riemann zeta-function.}  Mathematika 36 (1989): 71-88.




\bibitem{HardyWright} G. H. Hardy and E. M. Wright,
{\it An introduction to the theory of numbers}.
Sixth edition. Revised by D. R. Heath-Brown and J. H. Silverman. With a foreword by Andrew Wiles
Oxford University Press, Oxford, (2008). 


\bibitem{H} 
D. Hejhal, {\it On the distribution of $\log|\zeta'(1/2+it)|$}. in \textit{Number Theory, Trace Formulas, and Discrete Groups}, K. E. Aubert, E. Bombieri, and D. M. Goldfeld, eds., Proceedings of the 1987 Selberg Symposium, (Academic Press, 1989): 343--370.



\bibitem{HLZ}
W. Heap, J. Li, J. Zhao, {\it Lower bounds for discrete negative moments of the Riemann zeta function},
Algebra Number Theory 16 (2022), no. 7, 1589-1625. 


\bibitem{H}
C. P. 
Hughes, {\it 
Random matrix theory and discrete moments of the Riemann zeta function}.  
Random matrix theory. 
J. Phys. A 36 (2003), no. 12, 2907-2917. 

\bibitem{HKO} 
C. P. Hughes,  J. P. Keating, and N. O'Connell,  {\it Random matrix theory and the derivative of the Riemann zeta-function.} 
Proc. Roy. Soc. London A 456 (2000): 2611--2627. 

\bibitem{HPC}
C.  Hughes, A. Pearce-Crump, {\it 
A discrete mean-value theorem for the higher derivatives of the Riemann zeta function}. 
J. Number Theory 241 (2022), 142-164. 

\bibitem{Hum}
P. Humphries, {\it The distribution of weighted sums of the Liouville function and P\'{o}lya's conjecture}. J. Number Theory 133 (2013), no. 2, 545–582. 

\bibitem{In}
A. E. Ingham, {\it Mean-value theorems in the theory of the Riemann zeta function}, Proc.
London Math. Soc. (2) 27 (1926), 273-300.

\bibitem{IK}
H.  Iwaniec, E. Kowalski, {\it 
Analytic number theory}. 
American Mathematical Society Colloquium Publications, 53. American Mathematical Society, Providence, RI, 2004. 

\bibitem{Kad}
H. Kadiri, {\it Une r\'{e}gion explicite sans z\'{e}ros pour la fonction $\zeta$  de Riemann}. Acta Arith. 117 (2005), no. 4, 303-339.

\bibitem{KKY} D. Kaptan, Y. Karabulut, C. Yıldırım, {\it Some mean value theorems for the Riemann zeta-function and Dirichlet L-functions}, Comment. Math. Univ. St. Pauli 60(1–2) (2011) 83–87.
\bibitem{Yunus} Y. Karabulut, \textit{Some mean value problems concerning the Riemann zeta-function},  M.Sc Thesis, Boğaziçi University, (2009).

\bibitem{K}
S. Kirila,  {\it An upper bound for discrete moments of the derivative of the Riemann zeta-function}. Mathematika 66 (2020), no. 2, 475–497. 

\bibitem{Ko}
G. Kolesnik,
{\it On the order of Dirichlet L-functions.}
Pacific J. Math. 82 (1979), no. 2, 479-484. 


\bibitem{L}
N. Levinson, {\it More than one third of the zeros of Riemann's zeta-function are on $\sigma=1/2$}, 
Advances in Math. 13, (1974), 383-436.

\bibitem{LY}
 J. Liu and Y. Ye, {\it Perron's formula and the prime number theorem for automorphic L-functions}. Pure Appl. Math. Q. 3 (2007), no. 2, Special Issue: In honor of Leon Simon. Part 1, 481-497.

\bibitem{Me}
X. Meng,  {\it The distribution of k-free numbers and the derivative of the Riemann zeta-function}. Math. Proc. Cambridge Philos. Soc. 162 (2017), no. 2, 293–317. 

\bibitem{MN}
 M. B. Milinovich, N. Ng, {\it A note on a conjecture of Gonek.}  Funct. Approx. Comment. Math. 46 (2012), no. 2, 177--187.


\bibitem{MN2}
 M. B. Milinovich, N. Ng, {\it Lower bounds for moments of $\zeta'(\rho)$}. Int. Math. Res. Not. IMRN 2014, no. 12, 3190–3216.

\bibitem{MN3} 
M. B. Milinovich, N.  Ng,  {\it Simple zeros of modular L-functions}. Proc. Lond. Math. Soc. (3) 109 (2014), no. 6, 1465–1506.

\bibitem{Mo}
H. L. Montgomery, {\it Topics in multiplicative number theory}, Lecture Notes inMathematics 227 (Springer,
Berlin, 1971).


\bibitem{MV} H. L. Montgomery and R.C. Vaughan, {\it Hilbert's inequality}.  J. London Math. Soc. (2) 8 (1974): 73--82.

\bibitem{MV_book} H. L. Montgomery and R. C. Vaughan, {\it Multiplicative number theory. I. Classical theory}, Cambridge Stud. Adv. Math., 97
Cambridge University Press, Cambridge, 2007. xviii+552 pp.

\bibitem{Ng0} N. Ng, {\it The fourth moment of $\zeta'(\rho)$}. Duke Math. J. 125 (2004), no. 2, 243–266. 

\bibitem{Ng2}
 N. Ng, {\em The summatory function of the M\"{o}bius function}, Proc. London Math. Soc. (3) {\bf 89}, 2004, 361-389.

\bibitem{Ng}
N. Ng, {\it A discrete mean value of the derivative of the Riemann zeta function}. Mathematika 54 (2007), no. 1-2, 113–155.



\bibitem{Ng3}
N. Ng,  {\it Extreme values of $\zeta'(\rho)$}. J. Lond. Math. Soc. (2) 78 (2008), no. 2, 273–289.


\bibitem{S} P. Shiu, {\it A Brun-Titschmarsh theorem for multiplicative functions}.  Journal für die reine und angewandte Mathematik 313 (1980): 161-170.

\bibitem{T}
 E. C. Titchmarsh, {\it The theory of the Riemann zeta-function. Second edition}. Edited and with a preface by D. R. Heath-Brown. The Clarendon Press, Oxford University Press, New York, 1986. x+412 pp. 



  
\bibitem{Ts} K. M. Tsang, {\em Some $\Omega$-theorems for the Riemann zeta-function}, Acta Arith. 46 (1986), no. 4, 369--395. 





\end{thebibliography}
\end{document}